\documentclass[12pt,oneside,a4paper]{amsart}
\usepackage[utf8]{inputenc}
\usepackage[T1]{fontenc}
\usepackage{amssymb}
\usepackage{enumitem}
\usepackage{thmtools}
\usepackage{hyperref}
\usepackage{graphicx}
\usepackage{tikz}

\makeatletter
\def\@tocline#1#2#3#4#5#6#7{\relax
  \ifnum #1>\c@tocdepth 
  \else
    \par \addpenalty\@secpenalty\addvspace{#2}%
    \begingroup \hyphenpenalty\@M
    \@ifempty{#4}{%
      \@tempdima\csname r@tocindent\number#1\endcsname\relax
    }{%
      \@tempdima#4\relax
    }%
    \parindent\z@ \leftskip#3\relax \advance\leftskip\@tempdima\relax
    \rightskip\@pnumwidth plus4em \parfillskip-\@pnumwidth
    #5\leavevmode\hskip-\@tempdima
      \ifcase #1
       \or\or \hskip 1em \or \hskip 2em \else \hskip 3em \fi%
      #6\nobreak\relax
    \hfill\hbox to\@pnumwidth{\@tocpagenum{#7}}\par
    \nobreak
    \endgroup
  \fi}
\makeatother

\declaretheorem[numberwithin=section]{theorem}
\declaretheorem[numberlike=theorem]{lemma}

\declaretheorem[numberlike=theorem]{proposition}
\declaretheorem[numberlike=theorem]{algorithm}
\declaretheorem[style=definition,numberlike=theorem]{definition}
\declaretheorem[style=definition,numberlike=theorem]{example}
\declaretheorem[style=remark,numberlike=theorem]{remark}

\newcommand{\terminology}[1]{\emph{#1}}
\newcommand{\RR}{\mathbb{R}}

\newcommand{\NN}{\mathbb{N}}
\newcommand{\ZZ}{\mathbb{Z}}

\newcommand{\freecarnot}[1]{F_{#1}}
\newcommand{\freelie}[1]{\mathfrak{f}_{#1}}
\newcommand{\lielayer}[2]{#1^{[#2]}}
\newcommand{\freelielayer}[2]{\lielayer{\freelie{#1}}{#2}}
\newcommand{\lowercentralseriesterm}[2]{\lielayer{#1}{\geq #2}}

\newcommand{\abnormalpolynomial}[2]{P_{#1}^{#2}}
\newcommand{\Ad}[1]{\operatorname{Ad}_{#1}}
\newcommand{\ad}[1]{\operatorname{ad}_{#1}}
\newcommand{\abs}[1]{\left\vert #1\right\vert}
\newcommand{\freelieaction}{D}
\newcommand{\acts}{\curvearrowright}
\newcommand{\into}{\hookrightarrow}
\newcommand{\identity}[1]{e}

\newcommand{\hallset}{\mathcal{H}}
\newcommand{\polyring}[1]{\RR[#1]}
\newcommand{\fractionfield}[1]{\RR(#1)}
\newcommand{\polyringextended}[2]{\fractionfield{#1}[#2]}

\newcommand{\covector}{\lambda}
\newcommand{\covectorb}{\eta}
\newcommand{\differenceinteger}[1]{C_{#1}}
\newcommand{\poincare}{P}
\newcommand{\vectorparam}[1]{\boldsymbol{#1}}
\newcommand{\freeparam}{\alpha}
\newcommand{\factorpolynomial}[1]{S^{#1}}
\newcommand{\factorcoefficient}[1]{\nu_{#1}}
\newcommand{\lt}{<}
\newcommand{\gt}{>}
\newcommand{\amp}{&}
\title[ODE trajectories as abnormals]{ODE trajectories as abnormal curves in Carnot groups}
\author{Eero Hakavuori}
\address[Hakavuori]{SISSA, Via Bonomea 265, 34136 Trieste}
\email{eero.hakavuori@sissa.it}
\keywords{Carnot groups, sub-Riemannian geometry, abnormal extremals}
\date{June 17, 2020}
\begin{document}
\label{ode_abnormals}

\begin{abstract}
We prove that for every polynomial ODE there exists a Carnot group where the trajectories of the ODE lift to abnormal curves. The proof defines an explicit construction to determine a covector for the resulting abnormal curves. Using this method we give new examples of abnormal curves in Carnot groups of high step. As a byproduct of the argument, we also prove that concatenations of abnormal curves have abnormal lifts.
\end{abstract}

\subjclass[2010]{%
53C17, 
37N35, 
22E25. 
}

\maketitle
\tableofcontents\typeout{************************************************}
\typeout{Section 1 Introduction}
\typeout{************************************************}

\section{Introduction}\label{section-introduction}
\label{paragraph-carnot}
\label{g:notation:idp1}\label{g:notation:idp2} A Carnot group is a simply connected nilpotent Lie group \(G\) whose Lie algebra \(\mathfrak{g}\) admits a stratification \(\mathfrak{g}=\lielayer{\mathfrak{g}}{1}\oplus\dots\oplus\lielayer{\mathfrak{g}}{s}\), i.e.\@, a decomposition such that \([\lielayer{\mathfrak{g}}{1}, \lielayer{\mathfrak{g}}{i}] = \lielayer{\mathfrak{g}}{i+1}\) for all \(i=1,\ldots,s\), where \(\lielayer{\mathfrak{g}}{s+1}=\{0\}\). The dimension \(r=\dim\lielayer{\mathfrak{g}}{1}\) is called the \terminology{rank} of the Carnot group \(G\) and the largest integer \(s\) such that \(\lielayer{\mathfrak{g}}{s}\neq\{0\}\) is called the \terminology{step}. Let
\begin{equation*}
\operatorname{End}\colon L^2([0,1];\lielayer{\mathfrak{g}}{1}) \to G,\quad u\mapsto \gamma_u(1)
\end{equation*}
be the endpoint map, where \(\gamma_u\colon [0,1]\to G\) is the \terminology{horizontal curve} starting from the identity element \(\identity{G}\in G\) with control \(u\), i.e.\@, the unique absolutely continuous curve such that
\begin{equation*}
\frac{d}{dt}\gamma_u(t) = (L_{\gamma(t)})_*u(t)\quad\text{and}\quad \gamma_u(0) = \identity{G}\text{.}
\end{equation*}
Here \(L_{\gamma(t)}\colon G\to G\) is the left translation \(L_{\gamma(t)}(g) = \gamma(t)\cdot g\).

The \terminology{abnormal curves} are the trajectories \(\gamma_u\) for critical points \(u\in L^2([0,1];\lielayer{\mathfrak{g}}{1})\) of the endpoint map. Their significance arises from the Pontryagin Maximum Principle, which separates potentially length-minimizing curves in sub-Riemannian manifolds into two types: the normal and the abnormal extremals, see e.g. \cite[Section~3.4]{ABB-2020-subriemannian_geometry} for details. The normal extremals are well behaved, but the abnormal ones are the subject of two important open problems of sub-Riemannian geometry: the regularity of length minimizers and the Sard Conjecture, see \cite[Section~10]{Montgomery-2002-tour_of_subriemannian_geometries}.

The regularity problem asks what is the minimal regularity of sub-Riemannian length-minimizing curves. The issue is the existence of strictly abnormal minimizers, see \cite{Montgomery-1994-abnormal_minimizers}, \cite{Liu-Sussman-1995-shortest_paths_ranks_two}, \cite{Gole-Karidi-1995_carnot_geodesics} for some examples. The abnormal curves in general have no regularity beyond being Lipschitz, but nonetheless all known examples of abnormal length minimizers are \(C^\infty\)-smooth. Many partial regularity results for length minimizers exist, such as \(C^\infty\)-regularity in generic sub-Riemannian structures of rank at least \(3\) \cite{Chitour-Jean-Trelat-2006-genericity_results}, analyticity on an open dense subset of each minimizer \cite{Sussmann-2014-open_dense_regularity}, nonexistence of corner type singularities \cite{Hakavuori-Le_Donne-2016-corners}, and \(C^1\)-regularity in dimension 3 \cite{BFPR-2018-3d_Sard}.

The Sard Conjecture is that the set of critical values of the endpoint map should have zero measure, i.e.\@, most points should not be reachable from a fixed initial point with an abnormal curve. The set of critical values is known as the \terminology{abnormal set}. A more restricted variant is the Minimizing Sard Conjecture, where the abnormal curves are in addition required to be length minimizers. As with the regularity problem, several partial results exist. For instance, the minimizing abnormal set is contained in a closed nowhere dense set \cite{Agrachev-2009-points_of_smoothness}, the abnormal set is a proper algebraic or analytic subvariety in Carnot groups and polarized groups of particular types \cite{LMOPV-2016-sard}, and in dimension 3 the abnormal set is a countable union of semianalytic curves \cite{BFPR-2018-3d_Sard}.

Recently, both the regularity and Sard problems have seen progress through the study of abnormal curves from a dynamical systems viewpoint. A class of potentially minimizing abnormal curves in rank 2 sub-Riemannian structures was proved to have at least \(C^1\)-regularity \cite{BCJPS-2018-rank_2_abnormals}, and the Sard Conjecture was proved in Carnot groups of rank 2 step 4 and rank 3 step 3 \cite{Boarotto-Vittone-2020-dynamical_sard}. The idea common to both articles is that differentiating the identities defining abnormal curves leads to an ODE system that some reparametrization of the control of the abnormal curve will satisfy. Then using normal forms for the resulting ODEs, the authors of \cite{BCJPS-2018-rank_2_abnormals} and \cite{Boarotto-Vittone-2020-dynamical_sard} were able to prove their respective claims by studying trajectories of finitely many explicit ODE systems.

One reason for the success of the previous two articles was that the dynamical systems were not completely arbitrary, but had some properties that helped simplify the systems, such as the linear part being a traceless matrix. The goal of this paper is to study the scope of these dynamical systems within the setting of Carnot groups. The underlying questions motivating this research are:
\begin{itemize}[label=\textbullet]
\item{}Which dynamical systems can arise as the ODE systems of abnormal curves in Carnot groups?
\item{}Which curves in \(\RR^r\) are the horizontal projections of abnormal curves in Carnot groups of rank \(r\)?
\end{itemize}
The main result of this paper is that without fixing a specific Carnot group to study, there are essentially no restrictions on the possible dynamics of abnormal trajectories.
\begin{theorem}\label{theorem-every-ode-trajectory-is-abnormal}
Let \(P\colon G\to TG\) be a polynomial vector field in a Carnot group \(G\) of rank \(r\). Then there exists a Carnot group \(F\) of rank \(r\) and a surjective {Carnot group homomorphism} \(F\to G\) such that horizontal trajectories of the ODE \(\dot{x} = P(x)\) {lift} to abnormal curves in \(F\).
\end{theorem}
The Carnot group \(F\) in {Theorem~{\ref{theorem-every-ode-trajectory-is-abnormal}}} may be taken to be a free Carnot group whose step depends only on the degrees of the polynomials of the ODE \(P\), see {Theorem~{\ref{theorem-every-ode-trajectory-is-abnormal-quantified}}} for the more precise version. Applying {Theorem~{\ref{theorem-every-ode-trajectory-is-abnormal}}} to the simplest possible Carnot groups \(G=\RR^r\) shows that all trajectories of polynomial ODEs in \(\RR^r\) are the horizontal projections of abnormal curves of Carnot groups of rank \(r\). Several examples of this type will be presented in {Section~{\ref{section-examples}}}. {Theorem~{\ref{theorem-every-ode-trajectory-is-abnormal}}} can also be viewed as an algebraic complexity counterpart to the metric complexity result of \cite{Le_Donne-Zust-2019-inverse_limits} that arbitrary curves in \(\RR^r\) are well approximated by horizontal projections of length minimizing curves of rank \(r\) Carnot groups.

The ingredients used to prove {Theorem~{\ref{theorem-every-ode-trajectory-is-abnormal}}} also produce an auxiliary result that finite concatenations of abnormal curves have abnormal lifts.
\begin{theorem}\label{theorem-concatenation-of-abnormals}
Let \(G\) be a Carnot group of rank \(r\). Then there exists a Carnot group \(F\) of rank \(r\) and a surjective {Carnot group homomorphism} \(F\to G\) such that for all abnormal curves \(\alpha\) and \(\beta\) in \(G\) the concatenation \(\alpha\star\beta\) lifts to an abnormal curve in \(F\).
\end{theorem}
A potentially useful feature of the method of {Theorem~{\ref{theorem-every-ode-trajectory-is-abnormal}}} and {Theorem~{\ref{theorem-concatenation-of-abnormals}}} is that the abnormal lifts satisfy the Goh condition, see e.g. \cite[Section~12.3]{ABB-2020-subriemannian_geometry} for the definition. Moreover the groups \(F\) can be chosen such that the abnormal lifts also satisfy even higher order abnormality conditions such as the third order condition of \cite{Boarotto-Monti-Palmurella-2019-third_order}.

\typeout{************************************************}
\typeout{Subsection 1.1 Structure of the paper}
\typeout{************************************************}

\subsection{Structure of the paper}\label{g:subsection:idp3}
{Section~{\ref{section-abnormal-polynomials}}} lays the foundations needed to prove the main results. The characterization of abnormal curves in terms of abnormal polynomials is covered in {Subsection~{\ref{subsection-abnormal-polynomials}}}. {Subsection~{\ref{subsection-hall-sets}}} recalls the Hall basis construction for free Lie algebras and fixes the particular basis to be used in the rest of the paper. {Subsection~{\ref{subsection-polynomials-in-free-lie-algebras}}} describes how to use the fixed Hall basis to transfer abnormal polynomials between Carnot groups of different step by treating them as polynomials in a free Lie algebra. An important construction is the action of a free Lie algebra as derivations on its polynomial ring. Finally, {Subsection~{\ref{subsection-poincare-series}}} recalls the definition of a Poincaré series to facilitate counting the number of monomials of given degrees.

The core of the proofs of the main theorems is covered in {Section~{\ref{section-abnormality-of-ode-trajectories}}}. {Subsection~{\ref{subsection-reduction-to-searching-polynomial-factors}}} reduces abnormality of a trajectory of an ODE to the existence of a common factor in abnormal polynomials. {Subsection~{\ref{subsection-a-quotient-eliminating-higher-degree-variables}}} defines a quotient Lie algebra that eliminates irrelevant variables from the abnormal polynomials. The key argument is in {Subsection~{\ref{subsection-existence-of-arbitrary-polynomial-factors}}}, which rephrases finding a common factor as an eventually underdetermined linear system. The proofs of the main theorems are concluded in {Section~{\ref{section-main-proofs}}}.

The final part of the paper in {Section~{\ref{section-examples}}} is dedicated to constructing examples of abnormal curves. The proof of {Theorem~{\ref{theorem-every-ode-trajectory-is-abnormal}}} is condensed into an algorithm and several examples are computed, showcasing some possibilities and limitations of the algorithm.

\typeout{************************************************}
\typeout{Section 2 Abnormal polynomials}
\typeout{************************************************}

\section{Abnormal polynomials}\label{section-abnormal-polynomials}

\typeout{************************************************}
\typeout{Subsection 2.1 Abnormal curves and polynomials}
\typeout{************************************************}

\subsection{Abnormal curves and polynomials}\label{subsection-abnormal-polynomials}
In addition to the characterization as singular points of the endpoint map, abnormal curves can also be defined as characteristic curves of the canonical symplectic form on the cotangent bundle \cite{Hsu-1992-calc_var_via_griffiths_formalism}, see also \cite[Section~4.3.2]{ABB-2020-subriemannian_geometry}. In the setting of Carnot groups, right-trivialization of the cotangent bundle transfers this characterization from the cotangent bundle to the following description in terms of the Lie algebra, see \cite[Section~2.3]{LMOPV-2016-sard}.
\begin{definition}\label{definition-abnormal-polynomial}
\label{g:notation:idp4}
Let \(G\) be a Carnot group and \(\mathfrak{g}\) its Lie algebra. Let \(X\in\mathfrak{g}\) be a vector and \(\covector\in\mathfrak{g}^*\) a covector. The polynomial
\begin{equation*}
\abnormalpolynomial{X}{\covector}\colon G\to\RR,\quad \abnormalpolynomial{X}{\covector}(g) = \covector(\Ad{g}X)
\end{equation*}
is called an \terminology{abnormal polynomial}. Here \(\Ad{g}\in \mathrm{GL}(\mathfrak{g})\) is the adjoint map, i.e.\@ the differential of the conjugation map \(G\to G\), \(h\mapsto ghg^{-1}\).
\end{definition}
The abnormal polynomials were introduced with a different presentation in \cite{LLMV-2013-extremal_curves} and \cite{LLMV-2018-extremal_polynomials} under the name ``extremal polynomial'' for the purpose of characterizing abnormal curves. The presentation using the adjoint map was introduced in \cite{LMOPV-2016-sard}. The characterization in \cite[Theorem~1.1]{LLMV-2013-extremal_curves} and \cite[Corollary~2.14]{LMOPV-2016-sard} is the following.
\begin{lemma}\label{lemma-curves-in-abnormal-variety-are-abnormal}
A horizontal curve \(\gamma\) is abnormal if and only if there exists a nonzero \(\covector\in\mathfrak{g}^*\) such that \(\abnormalpolynomial{X}{\covector}\circ\gamma\equiv 0\) for all horizontal vectors \(X\in\lielayer{\mathfrak{g}}{1}\).
\end{lemma}
An important notion for curves in Carnot groups is the horizontal lift. This allows transferring curves between different Carnot groups of equal rank.
\begin{definition}\label{definition-projection}
Let \(G\) and \(F\) be two Carnot groups and \(\mathfrak{g}\) and \(\mathfrak{f}\) their Lie algebras. A Lie group homomorphism \(\pi\colon F\to G\) is a \terminology{Carnot group homomorphism} if \(\pi_*(\lielayer{\mathfrak{f}}{1})\subset \lielayer{\mathfrak{g}}{1}\).
\end{definition}
\begin{definition}\label{definition-lifted-curve}
Let \(G\) and \(F\) be Carnot groups and let \(\pi\colon F\to G\) be a {Carnot group homomorphism}. Let \(\gamma\colon(a,b)\to G\) be a horizontal curve in \(G\). A curve \(\tilde{\gamma}\colon (a,b)\to F\) is called a \terminology{horizontal lift} of the curve \(\gamma\), if \(\tilde{\gamma}\) is horizontal in \(F\) and \(\pi\circ\tilde{\gamma}=\gamma\).
\end{definition}
An important fact is that the horizontal lifts of abnormal curves are themselves also abnormal curves, see \cite[Proposition~2.27]{LMOPV-2016-sard}.

The proof of the main result {Theorem~{\ref{theorem-every-ode-trajectory-is-abnormal}}} will involve manipulating the abstract abnormal polynomials \(\abnormalpolynomial{X}{\covector}(g) = \covector(\Ad{g}X)\) in coordinates as concrete multivariate polynomials of a polynomial ring \(\polyring{x_1,\ldots,x_n}\). Since the adjoint map \(\Ad{}\colon G\to \mathrm{GL}(\mathfrak{g})\) is a Lie group homomorphism, a convenient system of coordinates is given by \terminology{exponential coordinates of the second kind}. For a basis \(X_1,\ldots,X_n\) of the Lie algebra \(\mathfrak{g}\), the exponential coordinates of the second kind are
\begin{equation*}
\RR^n\to G,\quad (x_1,\ldots,x_n)\mapsto \exp(x_nX_n)\cdots\exp(x_1X_1)\text{.}
\end{equation*}

\begin{lemma}\label{lemma-abnormal-polynomials-in-exponential-coordinates}
In exponential coordinates of the second kind with respect to the basis \(X_1,\dots,X_n\), the {abnormal polynomial} \(\abnormalpolynomial{X}{\covector}\) has the formula
\begin{equation*}
\abnormalpolynomial{X}{\covector}(x_1,\dots,x_n) = \sum_{i_1,\ldots,i_n}\frac{1}{i_1!\cdots i_n!}\covector\Big(\ad{X_{n}}^{i_n}\cdots \ad{X_{1}}^{i_1}X\Big)x_1^{i_1}\ldots x_n^{i_n}
\end{equation*}

\end{lemma}
\begin{proof}\label{g:proof:idp5}
The claim follows directly by applying the two formulas
\begin{equation*}
\Ad{\exp(X_i)\exp(X_j)} = \Ad{\exp(X_i)}\Ad{\exp(X_j)}
\end{equation*}
and
\begin{equation*}
\Ad{\exp(x_jX_{j})} = e^{\ad{x_jX_{j}}} = \sum_i\frac{1}{i!}x_j^i\ad{X_{j}}^i
\end{equation*}
to expand the expression \(\abnormalpolynomial{X}{\covector}(x_1,\ldots,x_n) = \covector(\Ad{\exp(x_nX_n)\cdots\exp(x_1X_1)}X)\).
\end{proof}

\typeout{************************************************}
\typeout{Subsection 2.2 Hall sets}
\typeout{************************************************}

\subsection{Hall sets}\label{subsection-hall-sets}
The coordinate formula of {Lemma~{\ref{lemma-abnormal-polynomials-in-exponential-coordinates}}} for the {abnormal polynomials} depends on the choice of the basis \(X_1,\ldots,X_n\) of the Lie algebra. In the case of free Lie algebras, possible bases are well understood in terms of \terminology{Hall sets}, see the original construction in \cite{Hall-1950-basis_for_free_lie_rings}. The following definition is the one in \cite[Section~4.1]{Reutenauer-1993-free_lie_algebras}, but with the trees structured as in \cite[Ch II, §2.10]{Bourbaki-1975-lie_chapter_1-3}.
\begin{definition}\label{definition-hall-set}
\label{g:notation:idp6}
\label{g:notation:idp7}
Let \(\mathcal{A}\) be a finite set and let \(M(\mathcal{A})\) be the free magma on \(\mathcal{A}\), i.e.\@, the set of all binary trees with leaves in \(\mathcal{A}\). For \(h_1,h_2\in M(\mathcal{A})\), denote by \((h_1,h_2)\in M(\mathcal{A})\) the binary tree whose left subtree is \(h_1\) and right subtree is \(h_2\).

A \terminology{Hall set} on \(\mathcal{A}\) is a totally ordered set \(\hallset\subset M(\mathcal{A})\) with the following properties:
\begin{enumerate}[label=(\roman*)]
\item\label{hall-def-singletons}All elements \(a\in \mathcal{A}\) are in \(\hallset\).
\item\label{hall-def-pairs}Every tree \(h=(h_1,h_2)\in \hallset\) satisfies \(h_1,h_2\in\hallset\) and \(h_1\lt h_2\) and \(h_1\lt h\).
\item\label{hall-def-triples}Every tree \((h_1,(h_{21},h_{22}))\in \hallset\) satisfies \(h_{21}\leq h_1\).
\item\label{hall-def-completness}Every tree \(h\in M(\mathcal{A})\) with the properties of \ref{hall-def-pairs} and \ref{hall-def-triples} is contained in \(\hallset\).
\end{enumerate}
The elements of \(\hallset\) are called \terminology{Hall trees}. The \terminology{degree} \(\deg(h)\) of a tree \(h\) is its number of leaves.
\end{definition}
\begin{remark}\label{remark-hall-order}
The order constraint \(h_1\lt h\) in \ref{hall-def-pairs} is sometimes imposed in the stronger form that \(\deg(h')\lt \deg(h)\) implies that \(h'\lt h\) for all \(h,h'\in\hallset\). This stronger condition is used in much of the older literature, such as the original construction of \cite{Hall-1950-basis_for_free_lie_rings} and in \cite{Bourbaki-1975-lie_chapter_1-3}. See the discussion in \cite[Section~4.5]{Reutenauer-1993-free_lie_algebras} for the history and reasons behind these different conditions.
\end{remark}
Viewing the elements \(\mathcal{A}\) as the generators of a free Lie algebra, each Hall tree specifies some iterated Lie bracket of the generators and hence an element of the free Lie algebra.  By \cite[Theorem~4.9]{Reutenauer-1993-free_lie_algebras}, any Hall set on \(\mathcal{A}\) defines a basis of the free Lie algebra on \(\mathcal{A}\).

Fixing a total order on the free magma \(M(\mathcal{A})\), a Hall set can be defined by a recursively enumerating the trees satisfying the conditions of {Definition~{\ref{definition-hall-set}}}, see \cite[Proposition~4.1]{Reutenauer-1993-free_lie_algebras}. The only requirement on the total order on the free magma is that \(h_1\lt h\) for any tree \(h=(h_1,h_2)\). For the purposes of this article, the order will also need to be compatible with the degree. The following recursively defined order will be used in the rest of the paper.
\begin{definition}\label{definition-deg-left-right-hall-set}
Let \((\mathcal{A},\lt)\) be a totally ordered set and let \(M(\mathcal{A})\) be the set of binary trees with leaves in \(\mathcal{A}\). The \terminology{deg-left-right order} on \(M(\mathcal{A})\) is the total order extending \(\lt\) such that \(h\lt h'\) for \(h,h'\in M(\mathcal{A})\) if any of the following hold:
\begin{enumerate}[label=(\roman*)]
\item{}\(\deg(h)\lt\deg(h')\),
\item{}\(\deg(h)=\deg(h')\) and \(h_1\lt h_1'\), or
\item{}\(\deg(h)=\deg(h')\) and \(h_1=h_1'\) and \(h_2\lt h_2'\),
\end{enumerate}
where
\begin{equation*}
h=\begin{cases}(h_1,h_2),\amp \text{if }\deg(h)\geq 2\\h_1,\amp \text{if }\deg(h)=1\end{cases}
\end{equation*}
and \(h_1',h_2'\) are similarly the left and right subtrees for \(h'\).

The \terminology{deg-left-right Hall set} for the free Lie algebra of rank \(r\) is the {Hall set} constructed with the deg-left-right order on \(M(\{1,2,\dots,r\})\), where \(1\lt 2\lt\cdots \lt r\) is ordered in the standard way.
\end{definition}
\label{g:notation:idp8} Each {Hall tree} \(h\in\hallset\subset M(\mathcal{A})\) has a corresponding \terminology{Hall word} \(w(h)\), which is a word with letters in \(A\). The Hall words are defined recursively for \(h=(h_1,h_2)\) by concatenations \(w(h) = w(h_1)w(h_2)\) starting from \(w(a)=a\) for each \(a\in \mathcal{A}\). The Hall words are in one to one correspondence with Hall trees, see \cite[Corollary~4.5]{Reutenauer-1993-free_lie_algebras}. This correspondence defines the deg-left-right order on the deg-left-right Hall words.
\begin{lemma}\label{lemma-hall-words-for-deg-left-right-order}
Let \(\hallset\) be any {Hall set} such that the order is compatible with the degree, i.e.\@, such that \(\deg(h)\lt \deg(h')\) implies \(h\lt h'\). Then the Hall words of degree \(\geq 2\) are exactly the words that factor as
\begin{equation*}
(w_{k})^{i_k}(w_{k-1})^{i_{k-1}}\cdots (w_{1})^{i_1}w\text{,}
\end{equation*}
where
\begin{enumerate}[label=(\roman*)]
\item\label{enum-deg-left-rightwords-exponents}\(i_1,\ldots,i_k\in \ZZ_+\) are arbitrary exponents, and
\item\label{enum-deg-left-rightwords-increasing-sequence}\(w_{1}\lt w_{2}\lt \cdots \lt w_{k}\lt w\) are Hall words of lower degree.
\item\label{enum-deg-left-rightwords-final-bracket}either \(\deg(w) = 1\) or there are Hall words \(w',w''\) such that \(w=w'w''\), \(w'\lt w''\) and \(w_{1}\geq w'\).
\end{enumerate}
Moreover, the {Hall tree} for the Hall word \((w_{k})^{i_k}(w_{k-1})^{i_{k-1}}\cdots (w_{1})^{i_1}w\) describes the Lie bracket
\begin{equation*}
\ad{X_{w_{k}}}^{i_k}\cdots \ad{X_{w_{1}}}^{i_1}X_{w}\text{.}
\end{equation*}

\end{lemma}
\begin{proof}\label{g:proof:idp9}
The claim that the Hall words factor as in the statement is straightforward. If \(h=(h_1,h_2)\) is a Hall tree of degree \(\geq 2\), then by {Definition~{\ref{definition-hall-set}}} \(h_1\lt h_2\) and if \(h_2=(h_{21},h_{22})\), then \(h_1\geq h_{21}\). Then \(w(h)\) admits the factorization \(w(h) = w(h_1)w(h_2)\) and \(\ad{X_{w(h_1)}}X_{w(h_2)}\) is the corresponding Lie bracket.

For the converse claim, let \(h\) be the binary tree described by a Lie bracket \(\ad{X_{w_{k}}}^{i_k}\cdots \ad{X_{w_{1}}}^{i_1}X_{w}\) with the words \(w_1,\ldots,w_k\), and \(w\) satisfying the conditions \ref{enum-deg-left-rightwords-exponents}\textendash{}\ref{enum-deg-left-rightwords-final-bracket}. Write the \(\ad{}\)-sequence without exponents as
\begin{equation*}
\ad{X_{w_{k}}}^{i_k}\cdots \ad{X_{w_{1}}}^{i_1}X_{w} = \ad{Y_{n}}\cdots\ad{Y_{1}}X_{w}\text{,}
\end{equation*}
where \(n=i_1+\cdots+i_k\), \(Y_1=\ldots=Y_{i_1}=X_{w_{1}}\) and so on. For each \(i=1,\ldots,k\), let \(h_i\) be the Hall tree corresponding to the Lie algebra element \(Y_i\) and let \(h_0\) be the Hall tree for the element \(X_{w}\), so that \(h=(h_n,(h_{n-1},(\ldots,(h_1,h_0)\ldots)))\). The claim that \(h\) is a Hall tree will follow by induction showing that every tree \(\bar{h}_i:=(h_i,(\ldots(h_1,h_0)\ldots))\) is a Hall tree. The case \(i=0\) holds by the choice of \(h_0=\bar{h}_0\).

In the case \(i=1\), assumption \ref{enum-deg-left-rightwords-increasing-sequence} implies that \(h_1\lt h_0\). If \(\deg(w)\gt 1\), let \(h'\) and \(h''\) be the Hall trees corresponding to the Hall words \(w'\) and \(w''\), so that \(h_0=(h',h'')\). Then assumption \ref{enum-deg-left-rightwords-final-bracket} implies that \(h_1\geq h'\). By assumption the order is compatible with the degree, so \(h_1\lt (h_1,h_0)=\bar{h}_1\). Hence the tree \(\bar{h}_1=(h_1,h_0)\) satisfies all the order constraints of {Definition~{\ref{definition-hall-set}}} and is a Hall tree.

Suppose the claim holds up to some \(i\geq 1\). The tree \(\bar{h}_{i+1} = (h_{i+1},(h_i,\bar{h}_{i-1}))\) then consists of Hall trees.  The degree first comparison guarantees that \(h_{i+1}\lt \bar{h}_{i}\lt \bar{h}_{i+1}\). Combined with assumption \ref{enum-deg-left-rightwords-increasing-sequence}, this guarantees that all of the order constraints \(h_{i+1}\lt (h_i,\bar{h}_{i-1})\), \(h_{i+1}\lt \bar{h}_{i+1}\), and \(h_i\leq h_{i+1}\) are satisfied, so \(\bar{h}_{i+1}\) is a Hall tree. By induction it follows that the tree \(h=\bar{h}_n\) is a Hall tree, proving the claim.
\end{proof}
\begin{example}\label{example-rank-2-deg-left-right-hall-set}
In rank 2, the deg-left-right Hall words up to degree 6 are
\begin{align*}
1\amp\lt 2\lt 12\lt 112\lt 212\lt 1112\lt 2112\lt 2212\\
\amp\lt 11112\lt 21112\lt 22112\lt 22212\lt 12112\lt 12212\\
\amp\lt 111112\lt 211112\lt 221112\lt 222112\lt 222212\\
\amp\lt 121112\lt 122112\lt 122212\lt 112212\text{.}
\end{align*}
For the words of degree 6, the longest possible factorizations in the form of {Lemma~{\ref{lemma-hall-words-for-deg-left-right-order}}} are
\begin{align*}
111112 \amp= (1)^5(2)\amp
222112 \amp= (2)^3(1)(12)\amp
122112 \amp= (12)(2)(112)\\
211112 \amp= (2)(1)^3(12)\amp
222212 \amp= (2)^4(12)\amp
122212 \amp= (12)(2)(212)\\
221112 \amp= (2)^2(1)^2(12)\amp
121112 \amp= (12)(1)(112)\amp
112212 \amp= (112)(212)\text{.}
\end{align*}

\end{example}

\typeout{************************************************}
\typeout{Subsection 2.3 Polynomials in free Lie algebras}
\typeout{************************************************}

\subsection{Polynomials in free Lie algebras}\label{subsection-polynomials-in-free-lie-algebras}
\label{g:notation:idp10}
In the proof of {Theorem~{\ref{theorem-every-ode-trajectory-is-abnormal}}} it will be convenient to consider {abnormal polynomials} not as polynomials in a single Carnot group \(G\), but as abstract multivariate polynomials that admit realizations as abnormal polynomials in several Carnot groups. Fixing a basis \(X_1,\ldots,X_n\) of the Lie algebra \(\mathfrak{g}\), the coordinate formula of {Lemma~{\ref{lemma-abnormal-polynomials-in-exponential-coordinates}}} realizes the identification of abnormal polynomials as multivariate polynomials \(\abnormalpolynomial{}{}\in \polyring{x_1,\ldots,x_n}\). Fixing a {Hall set} \(\hallset=\{w_{1},w_{2},\ldots\}\) for the free Lie algebra \(\freelie{r}\) of rank \(r\) then allows viewing multivariate polynomials \(\abnormalpolynomial{}{}\in\polyring{x_{w_{1}},x_{w_{2}},\ldots}\) as polynomials in several Carnot groups of rank \(r\) at once, as will be described next.
\begin{definition}\label{definition-carnot-group-compatible-with-hall-basis}
Let \(\hallset\) be a {Hall set} for the free Lie algebra \(\freelie{r}\) of rank \(r\). Let \(G\) be a Carnot group of rank \(r\), so its Lie algebra is the quotient \(\mathfrak{g}=\freelie{r}/I\) of the free Lie algebra \(\freelie{r}\) by some ideal \(I\subset\freelie{r}\). The Carnot group \(G\) is said to be \terminology{compatible with the Hall set \(\hallset\)}, if the ideal \(I\) has a basis consisting of elements of the Hall set.
\end{definition}
\begin{definition}\label{definition-hall-exponential-coordinates}
Let \(\hallset\) be a {Hall set} for a free Lie algebra \(\freelie{r}\) and let \(G\) be a Carnot group of rank \(r\) that is {compatible} with the Hall set \(\hallset\). Let \(w_{1},\dots,w_{n}\) be the complementary Hall words to the ideal defining \(\mathfrak{g}\), so that \(\mathfrak{g}=\operatorname{span}\{X_{w_{1}},\dots,X_{w_{n}}\}\). The exponential coordinates of the second kind
\begin{equation*}
\RR^n\to G,\quad (x_1,\dots,x_n) \mapsto \exp(x_nX_{w_{n}})\cdots\exp(x_1X_{w_{1}})
\end{equation*}
are said to be \terminology{adapted to the Hall set} if \(w_{n}\gt\ldots\gt w_{1}\).
\end{definition}
\begin{definition}\label{definition-polynomial-in-free-lie-algebra}
\label{g:notation:idp11}
Let \(\hallset=\{w_{1},w_{2},\ldots\}\) be a {Hall set} for a free Lie algebra \(\freelie{r}\). Let \(\polyring{\hallset}:=\polyring{x_{w_{1}},x_{w_{2}},\ldots}\) be the weighted polynomial ring with a countable number of generators, where the weight of the variable \(x_{w_{i}}\) is the degree of the Hall word \(w_{i}\).
\end{definition}
The above definitions allow identifying each polynomial in a Carnot group compatible with a Hall set \(\hallset\) with a unique polynomial in \(\polyring{\hallset}\). {Abnormal polynomials} are not completely arbitrary polynomials however, so it will be relevant that the coordinate form of abnormal polynomials is preserved by lifting to higher step groups.
\begin{lemma}\label{lemma-lifting-quotient-abnormal-polynomials}
Let \(\mathfrak{f}\) be a stratified Lie algebra with a basis \(X_1,\ldots, X_n\). Let \(\mathfrak{g}\) be a stratified Lie algebra defined as a quotient of \(\mathfrak{f}\) by an ideal \(I = \operatorname{span}\{X_{m+1},\ldots,X_n\}\). Let \(F\) and \(G\) be Carnot groups with Lie algebras \(\mathfrak{f}\) and \(\mathfrak{g}\) respectively. Then for every vector \(X\in\mathfrak{g}\) and covector \(\covector\in \mathfrak{g}^*\) there exist lifts \(\tilde{X}\in\mathfrak{f}\) and \(\tilde{\covector}\in\mathfrak{f}^*\) such that the {abnormal polynomials} \(\abnormalpolynomial{X}{\covector}\) in \(G\) and \(\abnormalpolynomial{\tilde{X}}{\tilde{\covector}}\) in \(F\) have identical {coordinate expressions} in exponential coordinates of the second kind with respect to the bases \(X_1,\ldots,X_n\) and \(\pi(X_1),\ldots,\pi(X_m)\).
\end{lemma}
\begin{proof}\label{g:proof:idp12}
The lift of the covector is the pullback \(\tilde{\covector}:=\pi^*\covector\). The vector \(X\) is lifted by setting the coefficients of the elements \(X_{m+1},\ldots,X_n\) to zero. That is, if \(X = \sum_{i=1}^{m}x_i\pi(X_i)\), then the lift is \(\tilde{X} = \sum_{i=1}^{m}x_iX_i\).

The projection \(\pi\colon\mathfrak{f}\to\mathfrak{g}\) is a Lie algebra homomorphism, so the definitions of the lifts imply that
\begin{equation*}
\tilde{\covector}\Big(\ad{X_{n}}^{i_n}\cdots \ad{X_{1}}^{i_1}\tilde{X}\Big)
= \covector\Big(\ad{\pi(X_{n})}^{i_n}\cdots \ad{\pi(X_{1})}^{i_1}X\Big)\text{.}
\end{equation*}
Since \(\pi(X_{n})=\ldots=\pi(X_{m+1})=0\), the formula of {Lemma~{\ref{lemma-abnormal-polynomials-in-exponential-coordinates}}} implies that the abnormal polynomial \(\abnormalpolynomial{\tilde{X}}{\tilde{\covector}}\) in \(G\) only contains monomials with variables \(x_1,\ldots,x_m\). Moreover by the same formula, the coefficient of each monomial containing only variables \(x_1,\ldots,x_m\) is exactly the coefficient of the same monomial in the abnormal polynomial \(\abnormalpolynomial{X}{\covector}\) in \(H\).
\end{proof}
In a Carnot group \(G\), there is a natural action \(\mathfrak{g}\acts C^\infty(G)\) of the Lie algebra \(\mathfrak{g}\) on the space of smooth functions \(G\to\RR\), defined by extending a vector \(X\in\mathfrak{g}\) to a left-invariant vector field \(\tilde{X}(g) = (L_{g})_*X\) and using the action of vector fields as derivations on the space of smooth functions. In coordinates this defines an action \(\mathfrak{g}\acts \polyring{x_1,\ldots,x_n}\) by derivations also on the space of polynomials.
\begin{definition}\label{definition-free-lie-action-on-polynomials}
\label{g:notation:idp13}
Let \(\hallset\) be a {Hall set} for a free Lie algebra \(\freelie{r}\). Define an action \(\freelieaction\colon \freelie{r}\acts\polyring{\hallset}\) via derivations by extending all the actions \(\mathfrak{g}\acts \polyring{x_1,\ldots,x_n}\) as follows:

Let \(X\in\freelie{r}\) be a vector and \(P\in\polyring{\hallset}\) a polynomial. Let \(x_{w_{1}},\ldots,x_{w_{k}}\) be all the variables of \(P\) and let \(G\) be any Carnot group {compatible} with the Hall set \(\hallset\) such that none of the basis elements \(X_{w_{1}},\ldots,X_{w_{k}}\) are quotiented away. Fix exponential coordinates \(\RR^n\to G\) {adapted} to the Hall set \(\hallset\), and let \(\iota\colon\polyring{x_1,\ldots,x_n}\into \polyring{\hallset}\) be the identification of polynomials in \(G\) in coordinates with polynomials in \(\polyring{\hallset}\). The action \(\freelieaction\colon \freelie{r}\acts\polyring{\hallset}\) is defined by
\begin{equation*}
\freelieaction\colon\freelie{r}\times\polyring{\hallset}\to\polyring{\hallset},\quad \freelieaction(X,P) := \iota\Big(\pi(X)\iota^{-1}(P)\Big)\text{,}
\end{equation*}
where \(\pi\colon \freelie{r}\to\mathfrak{g}\) is the quotient projection, and \(\pi(X)\iota^{-1}(P)\) is defined by the action \(\mathfrak{g}\acts C^\infty(G)\).
\end{definition}
The well-posedness of {Definition~{\ref{definition-free-lie-action-on-polynomials}}} critically relies on the assumption that all the coordinates on the Carnot groups involved are adapted to a single fixed {Hall set}. Without this assumption, the coordinate expressions of the actions \(\mathfrak{g}\acts\polyring{x_1,\ldots,x_n}\) are in general incompatible with each other.

An extremely useful feature of the abnormal polynomials is how their behavior under the action \(\mathfrak{g}\acts\polyring{x_1,\ldots,x_n}\) is linked with the structure of the Lie algebra \(\mathfrak{g}\). The following result is a rephrasing of \cite[Theorem~1.1]{LLMV-2018-extremal_polynomials} and \cite[Proposition~2.22]{LMOPV-2016-sard} using the linearity of the abnormal polynomials \(\abnormalpolynomial{Y}{\covector}\) in the vector \(Y\).
\begin{lemma}\label{lemma-derivatives-of-abnormal-polynomials}
Let \(G\) be a Carnot group and \(\mathfrak{g}\) its Lie algebra. Then for every covector \(\covector\in\mathfrak{g}^*\) and all vectors \(X,Y\in\mathfrak{g}\), the \(X\)-derivative of the {abnormal polynomial} \(\abnormalpolynomial{Y}{\covector}\) is the abnormal polynomial \(X\abnormalpolynomial{Y}{\covector} = \abnormalpolynomial{[X,Y]}{\covector}\).
\end{lemma}
In the proof of {Theorem~{\ref{theorem-every-ode-trajectory-is-abnormal}}} it will be necessary to further simplify the monomial coefficients of abnormal polynomials given by the formula of {Lemma~{\ref{lemma-abnormal-polynomials-in-exponential-coordinates}}}. The following lemma is a criterion for when these coefficients are as simple as possible for the deg-left-right Hall set.
\begin{lemma}\label{lemma-free-monomials-of-abnormal-polynomials}
Let \(w\) be a {deg-left-right Hall word} with factorization \(w=w'w''\) and let \(x^I = x_{v_1}^{i_1}\ldots x_{v_\ell}^{i_\ell}\) be a monomial in \(\polyring{\hallset}\) such that \(v_1\lt\ldots\lt v_\ell\lt w\) and \(v_1\geq w'\). Then \(vw:=(v_\ell)^{i_\ell}\cdots (v_1)^{i_1}w\) is a deg-left-right Hall word and in exponential coordinates {adapted} to the deg-left-right order the coefficient of the monomial \(x^I\) in the {abnormal polynomial} \(\abnormalpolynomial{w}{\covector}\) is \(\frac{1}{i_1!\cdots i_\ell!}\covector_{vw}\).
\end{lemma}
\begin{proof}\label{g:proof:idp14}
\label{g:notation:idp15} By the explicit formula of {Lemma~{\ref{lemma-abnormal-polynomials-in-exponential-coordinates}}} the coefficient of the monomial \(x^I\) in the abnormal polynomial \(\abnormalpolynomial{w}{\covector}\) is
\begin{equation*}
\frac{1}{i_1!\cdots i_\ell!}\covector\Big(\ad{X_{v_\ell}}^{i_\ell}\cdots \ad{X_{v_1}}^{i_1}X_{w}\Big)\text{.}
\end{equation*}
By {Lemma~{\ref{lemma-hall-words-for-deg-left-right-order}}}, the assumption that \(v_1\lt\ldots\lt v_\ell\lt w\) and \(v_1\geq w'\) guarantees that \(vw\) is a deg-left-right Hall word. The second part of the lemma states the bracket \(\ad{X_{v_\ell}}^{i_\ell}\cdots \ad{X_{v_1}}^{i_1}X_{w}\) is exactly the Hall basis element \(X_{vw}\), so
\begin{equation*}
\covector\Big(\ad{X_{v_\ell}}^{i_\ell}\cdots \ad{X_{v_1}}^{i_1}X_{w}\Big) = \covector(X_{vw}) = \covector_{vw}\text{.}\qedhere
\end{equation*}

\end{proof}

\typeout{************************************************}
\typeout{Subsection 2.4 Poincaré series}
\typeout{************************************************}

\subsection{Poincaré series}\label{subsection-poincare-series}
The proof of {Theorem~{\ref{theorem-every-ode-trajectory-is-abnormal}}} involves counting the number of monomials of a given degree depending on some subset of the variables \(x_{w}\) in \(\polyring{\hallset}\).  A convenient form to manipulate these quantities will be in terms of generating functions, i.e.\@, in terms of Poincaré series.
\begin{definition}\label{definition-poincare-series}
Let \(W = \bigoplus_{i=0}^\infty W_i\) be a graded vector space with all components \(W_i\) finite dimensional. The \terminology{Poincaré series} of the graded vector space \(W\) is the formal power series
\begin{equation*}
\poincare(t) = \sum_{k=0}^\infty (\dim W_k)t^k\text{.}
\end{equation*}

\end{definition}
The Poincaré series of a weighted polynomial ring has a rather simple expression as a rational function depending on the weights.
\begin{lemma}\label{lemma-poincare-series-for-monomials-of-bounded-degree}
Let \(r\in\NN\) and \(d\in\NN\). Let \(w_{1}\lt\dots\lt w_{n}\) be all the {Hall words} of degree up to \(d\) in the free Lie algebra \(\freelie{r}\) of rank \(r\), and let \(d_{k}:=\deg(w_{k})\). The Poincaré series for the weighted polynomial ring \(\polyring{x_{w_{1}},\ldots,x_{w_{n}}}\) with weights \(d_{1},\ldots,d_{n}\) is given by the rational function
\begin{equation*}
\poincare(t) = 1/\prod_{k=1}^{n} (1-t^{d_{k}}) = 1/\prod_{k=1}^d (1-t^k)^{\dim\freelielayer{r}{k}}\text{.}
\end{equation*}

\end{lemma}
\begin{proof}\label{g:proof:idp16}
By \cite[Ch. V, §5.1, Proposition~1]{Bourbaki-1968-lie_chapitre_4-6}, the Poincaré series for the weighted polynomial ring is given by the first formula \(\poincare(t) = 1/\prod_{k=1}^{n} (1-t^{d_{k}})\). Since there are exactly \(\dim\freelielayer{r}{k}\) Hall words of degree \(k\), the second formula follows by grouping terms.
\end{proof}

\typeout{************************************************}
\typeout{Section 3 Common factors of abnormal polynomials}
\typeout{************************************************}

\section{Common factors of abnormal polynomials}\label{section-abnormality-of-ode-trajectories}

\typeout{************************************************}
\typeout{Subsection 3.1 Link between ODEs and polynomial factors}
\typeout{************************************************}

\subsection{Link between ODEs and polynomial factors}\label{subsection-reduction-to-searching-polynomial-factors}
The {characterization} of abnormal curves in Carnot groups as curves contained in the abnormal varieties leads to a strategy to prove {Theorem~{\ref{theorem-every-ode-trajectory-is-abnormal}}} about abnormality of ODE trajectories by searching for common factors in abnormal polynomials. The first part is {Lemma~{\ref{lemma-vector-of-polynomials-as-derivatives-in-free-lie-algebra}}}, which states that given an arbitrary set of polynomials \(Q_{1},\ldots,Q_{r}\), there exists some polynomial \(Q\) in high enough step whose horizontal derivatives are \(X_iQ=Q_{i}\). This gives a method to construct a polynomial first integral \(Q\) for any polynomial ODE.

The second part is {Lemma~{\ref{lemma-curves-annihilating-higher-order-abnormal-polynomials}}}, which states that for an arbitrary layer \(m\lt s\), annihilating all abnormal polynomials \(\abnormalpolynomial{X}{\covector}\) with \(X\in \lielayer{\mathfrak{g}}{m}\) is a sufficient criterion for abnormality. Hence if all the abnormal polynomials of a given layer have a common factor \(Q\), then all trajectories within the variety \(Q=0\) are abnormal.
\begin{lemma}\label{lemma-vector-of-polynomials-as-derivatives-in-free-lie-algebra}
Let \(\hallset\) be a {Hall set} for the free Lie algebra \(\freelie{r}\) of rank \(r\) and let \(X_1,\ldots,X_r\) be the generators of \(\freelie{r}\). Then for every collection of polynomials \(Q_{1},\ldots,Q_{r}\in\polyring{\hallset}\) there exists a polynomial \(Q\in\polyring{\hallset}\) such that \(X_iQ = Q_{i}\) for all \(i=1,\ldots,r\).
\end{lemma}
\begin{proof}\label{g:proof:idp17}
Define a linear map \(\varphi\colon\freelie{r}\to\polyring{\hallset}\) recursively by
\begin{align}
\varphi(X_i) \amp = Q_{i},\quad i=1,\ldots,r\notag\\
\varphi([X,Y]) \amp = X\varphi(Y)-Y\varphi(X)\text{,}\label{eq-recursive-def-varphi}
\end{align}
where \(X\varphi(Y)\) and \(Y\varphi(X)\) are determined by the derivation action \(\freelie{r}\acts\polyring{\hallset}\) as defined in {Definition~{\ref{definition-free-lie-action-on-polynomials}}}. A routine computation with the recursion formula {({\ref{eq-recursive-def-varphi}})} shows that \(\varphi([X,Y])=-\varphi([Y,X])\) and
\begin{equation*}
\varphi([X,[Y,Z]])+\varphi([Y,[Z,X]])+\varphi([Z,[X,Y]])=0\text{.}
\end{equation*}
Since there are no relations in the free Lie algebra except for anticommutativity and the Jacobi identity, it follows that the map \(\varphi\) is well defined.

The action \(\freelie{r}\acts\polyring{\hallset}\) is by derivations, so polynomials \(X\varphi(Y)\) are of lower (weighted) degree than polynomials \(\varphi(Y)\). In particular, \(\varphi(\freelielayer{r}{k})=0\) for all \(k\gt \max_i\deg(Q_i)+1\). Let \(s\geq 1\) be the minimal index such that \(\varphi(\freelielayer{r}{k})=0\) for all \(k\gt s\).

\label{g:notation:idp18}\label{g:notation:idp19} By the choice of the step \(s\), the map \(\varphi\) induces a well defined linear map \(\freelie{r,s}\to\polyring{x_{w_{1}},\ldots,x_{w_{n}}}\), where \(\freelie{r,s}\) is the free nilpotent Lie algebra of rank \(r\) and step \(s\), and \(w_{1},\ldots,w_{n}\) is the complete list of Hall words of degrees \(1,\ldots,s\). For ease of notation, the restricted map will also be denoted by \(\varphi\).

In the free Carnot group \(\freecarnot{r,s}\) of rank \(r\) and step \(s\), consider the differential one-form \(\omega\) such that for any left-invariant vector field \(\tilde{X}(g) = (L_{g})_*X\), the function \(g\mapsto \omega(\tilde{X}(g))\) is exactly the function \(\varphi(X)\) when computed in exponential coordinates {adapted} to the Hall set. That is, in these coordinates the one-form is defined for any vector \(X\in T_x\freecarnot{r,s}\) by
\begin{equation*}
\omega_{x}(X) := \varphi((L_{x^{-1}})_*X)(x)\text{.}
\end{equation*}

For vector fields \(\tilde{X},\tilde{Y}\colon \freecarnot{r,s}\to T\freecarnot{r,s}\), the classical formula for the differential of a one-form states that
\begin{equation*}
d\omega(\tilde{X},\tilde{Y}) = \tilde{X}\omega(\tilde{Y})-\tilde{Y}\omega(\tilde{X})-\omega([\tilde{X},\tilde{Y}])\text{.}
\end{equation*}
For left-invariant vector fields \(\tilde{X}(x)=(L_x)_*X\) and \(\tilde{Y}(x)=(L_x)_*Y\) the above simplifies to
\begin{equation*}
d\omega(\tilde{X},\tilde{Y}) = X\varphi(Y)-Y\varphi(X)-\varphi([X,Y])\text{.}
\end{equation*}
By the recursive definition {({\ref{eq-recursive-def-varphi}})} of the map \(\varphi\), it follows that \(d\omega = 0\), i.e.\@, the form \(\omega\) is closed. Since \(\freecarnot{r,s}\) is simply connected as a Carnot group, the closed form \(\omega\) is exact, so there exists a smooth function \(Q\colon\freecarnot{r,s}\to\RR\) such that \(dQ = \omega\). The coefficients of the one-form \(\omega\) are all polynomial, so the smooth function \(Q\) is a polynomial. Moreover, the polynomial \(Q\) has the desired property
\begin{equation*}
X_iQ = dQ(\tilde{X}_i) = \omega(\tilde{X}_i) = \varphi(X_i) = Q_{i}\text{.}\qedhere
\end{equation*}

\end{proof}
\begin{lemma}\label{lemma-curves-annihilating-higher-order-abnormal-polynomials}
Let \(G\) be a Carnot group of step \(s\) and let \(1\leq m\lt s\). Let \(\gamma\colon (a,b)\to G\) be a horizontal curve whose closure \(\overline{\gamma((a,b))}\subset G\) contains the identity \(\identity{G}\) of the group \(G\). If there exists a covector \(\covector\in\mathfrak{g}^*\) such that \(\abnormalpolynomial{X}{\covector}\circ\gamma\equiv 0\) for all \(X\in \lielayer{\mathfrak{g}}{m}\) and at least one of the {abnormal polynomials} \(\abnormalpolynomial{X}{\covector}\) is nonzero, then the curve \(\gamma\) is abnormal.
\end{lemma}
\begin{proof}\label{g:proof:idp20}
The claim will be proved by induction on the layer \(m\). The base case \(m=1\) follows from the {characterization} of abnormality by abnormal polynomials. Suppose that the claim holds up to some layer \(m\lt s-1\) and suppose that
\begin{equation}
\abnormalpolynomial{X}{\covector}\circ\gamma\equiv 0\label{eq-abnormal-vanishing-layer}
\end{equation}
for every \(X\in \lielayer{\mathfrak{g}}{m+1}\), and that at least one of the polynomials \(\abnormalpolynomial{X}{\covector}\) is nonzero.

By the explicit formula of {Lemma~{\ref{lemma-abnormal-polynomials-in-exponential-coordinates}}}, the abnormal polynomials \(\abnormalpolynomial{X}{\covector}\) for \(X\in \lielayer{\mathfrak{g}}{m+1}\) do not depend on the components of degree less than \(m+1\) of the covector \(\covector\). Hence without loss of generality it will be assumed that \(\covector(\lielayer{\mathfrak{g}}{m}) = 0\).

Fix an arbitrary \(Y\in\lielayer{\mathfrak{g}}{m}\). By {Lemma~{\ref{lemma-derivatives-of-abnormal-polynomials}}}, the horizontal derivatives for \(X\in\lielayer{\mathfrak{g}}{1}\) of the abnormal polynomial \(\abnormalpolynomial{Y}{\covector}\) are abnormal polynomials \(\abnormalpolynomial{[X,Y]}{\covector}\). Since \([X,Y]\in \lielayer{\mathfrak{g}}{m+1}\), assumption {({\ref{eq-abnormal-vanishing-layer}})} implies that the abnormal polynomial \(\abnormalpolynomial{Y}{\covector}\) is constant along the horizontal curve \(\gamma\).

The property \(\covector(\lielayer{\mathfrak{g}}{m}) = 0\) implies that \(\abnormalpolynomial{Y}{\covector}(\identity{G}) = 0\). Moreover, by assumption \(\identity{G}\in\overline{\gamma((a,b))}\), so by continuity of the polynomial \(\abnormalpolynomial{Y}{\covector}\), it follows that \(\abnormalpolynomial{Y}{\covector}\circ\gamma\equiv 0\). Since \(Y\in\lielayer{\mathfrak{g}}{m}\) was arbitrary, the inductive step is complete, and the claim follows.
\end{proof}
\begin{remark}\label{remark-nonzero-polynomial-vs-covector}
\label{g:notation:idp21}
The condition in {Lemma~{\ref{lemma-curves-annihilating-higher-order-abnormal-polynomials}}} that at least one of the {abnormal polynomials} \(\abnormalpolynomial{X}{\covector}\) for \(X\in\lielayer{\mathfrak{g}}{m}\) is nonzero can be equivalently rephrased as the requirement that \(\covector(\lowercentralseriesterm{\mathfrak{g}}{m})\neq 0\), where \(\lowercentralseriesterm{\mathfrak{g}}{m} = \lielayer{\mathfrak{g}}{m}\oplus\cdots\oplus\lielayer{\mathfrak{g}}{s}\) is the \(m\):th term of the lower central series of \(\mathfrak{g}\).
\end{remark}

\typeout{************************************************}
\typeout{Subsection 3.2 A quotient eliminating higher degree variables}
\typeout{************************************************}

\subsection{A quotient eliminating higher degree variables}\label{subsection-a-quotient-eliminating-higher-degree-variables}
Not all monomials appear in every {abnormal polynomial} even in arbitrarily high step. For instance, the coefficient of the monomial \(x_1\) in \(\abnormalpolynomial{X_1}{\covector}\) is always \(\covector(\ad{X_1}X_1) = 0\). For a vector \(X\) in layer \(m\) of a free Carnot group of step \(s\), every monomial of degree at most \(s-m\) containing only variables of degrees \(1,\ldots,m-1\) appears in the abnormal polynomial \(\abnormalpolynomial{X}{\covector}\) for some covector \(\covector\). For monomials involving variables of degree \(m\) and above there is no such guarantee. For this reason, the variables of degree \(m\) and above will not be useful for the search of common factors in abnormal polynomials, and it will be more convenient to eliminate them completely.
\begin{lemma}\label{lemma-quotient-by-higher-variables}
Let \(\freelie{r,s}\) be the free nilpotent Lie algebra of rank \(r\) and step \(s\). Let \(m\in\NN\) and let \(\mathfrak{g}\) be the stratified Lie algebra defined by quotienting \(\freelie{r,s}\) by the ideal \([\lowercentralseriesterm{\freelie{r,s}}{m},\lowercentralseriesterm{\freelie{r,s}}{m}]\), where \(\lowercentralseriesterm{\freelie{r,s}}{m} = \freelielayer{r,s}{m}\oplus \freelielayer{r,s}{m+1} \oplus \cdots \freelielayer{r,s}{s}\) is the \(m\):th term of the lower central series of \(\freelie{r,s}\). Then in the Carnot group \(G\) with Lie algebra \(\mathfrak{g}\), every {abnormal polynomial} \(\abnormalpolynomial{X}{\covector}\) with \(X\in\lielayer{\mathfrak{g}}{m}\) depends only on the variables of degrees \(1,\ldots, m-1\).
\end{lemma}
\begin{proof}\label{g:proof:idp22}
By the formula of {Lemma~{\ref{lemma-abnormal-polynomials-in-exponential-coordinates}}}, if some abnormal polynomial \(\abnormalpolynomial{X}{\covector}\) contains a monomial \(x^I=x_1^{i_1}\cdots x_n^{i_n}\), then there exists a nonzero bracket
\begin{equation*}
\ad{X_{n}}^{i_n}\cdots \ad{X_{1}}^{i_1}X \neq 0
\end{equation*}
in the Lie algebra \(\mathfrak{g}\).

Let \(x^I\) be any monomial that contains a variable \(x_j\) of degree \(\deg(x_j)\geq m\), so \(i_j\gt 0\). By assumption \(X_j\in \lowercentralseriesterm{\mathfrak{g}}{m}\) and \(X\in \lielayer{\mathfrak{g}}{m}\), so
\begin{equation*}
\ad{X_j}\ad{X_{j-1}}^{i_{j-1}}\cdots\ad{X_{1}}^{i_1}X = [X_j,\ad{X_{j-1}}^{i_{j-1}}\cdots\ad{X_{1}}^{i_1}X]\in [\lowercentralseriesterm{\mathfrak{g}}{m},\lowercentralseriesterm{\mathfrak{g}}{m}]\text{.}
\end{equation*}
By the construction of the Lie algebra \(\mathfrak{g}\), the above bracket must vanish. But then the longer bracket \(\ad{X_{n}}^{i_n}\cdots \ad{X_{1}}^{i_1}X\) is also zero, so the monomial \(x^I\) cannot appear in any abnormal polynomial \(\abnormalpolynomial{X}{\covector}\) with \(X\in\lielayer{\mathfrak{g}}{m}\).
\end{proof}

\typeout{************************************************}
\typeout{Subsection 3.3 The linear system for abnormal polynomial factors}
\typeout{************************************************}

\subsection{The linear system for abnormal polynomial factors}\label{subsection-existence-of-arbitrary-polynomial-factors}
For a fixed polynomial \(Q\) and layer \(m\in\NN\), the existence of a covector \(\covector\in\mathfrak{g}^*\) such that \(Q\) is a common factor for all the {abnormal polynomials} \(\abnormalpolynomial{X}{\covector}\), \(X\in \lielayer{\mathfrak{g}}{m}\), defines a linear system as will be described in the proof of {Proposition~{\ref{proposition-existence-of-polynomial-factor}}}. The linear system can be reduced to a form where the number of equations grows slower than the number of variables when increasing the step. Hence for a large enough step, the resulting linear system is underdetermined and has nontrivial solutions.

One reason for the slower increase of the number of equations is the following lemma showing that at least one of the abnormal polynomials in each layer can be freely chosen.
\begin{lemma}\label{lemma-free-abnormal-polynomial}
Let \(m\geq 2\) be a fixed layer and \(s\in\NN\) a fixed step. Let \(w\) be the minimal {deg-left-right Hall word} of degree \(m\), i.e.\@, the Hall word \(w=1^{m-1}2\). Let \(G\) be the Carnot group whose Lie algebra is the quotient \(\mathfrak{g} = \freelie{r,s}/[\lowercentralseriesterm{\freelie{r,s}}{m},\lowercentralseriesterm{\freelie{r,s}}{m}]\) eliminating variables of degree \(m\) and higher. Then for every polynomial \(P\colon G\to\RR\) of degree at most \(s-m\) involving only variables of degrees \(1,\ldots,m-1\), there exists a covector \(\covector\in\mathfrak{g}^*\) such that \(\abnormalpolynomial{w}{\covector} = P\).
\end{lemma}
\begin{proof}\label{g:proof:idp23}
For each monomial \(x^I\), let \(x_{v_1}^{i_1}\cdots x_{v_\ell}^{i_\ell} = x^I\) be the decomposition such that \(v_1\lt\cdots\lt v_\ell\). Denote by \(v(I) := (v_\ell)^{i_\ell}\cdots (v_1)^{i_1}\) the resulting word. The factorization of the minimal Hall word \(w=1^{m-1}2\) into two Hall words is \(w=(1)(1^{m-2}2)\). Since \(1\) is the minimal deg-left-right Hall word, every monomial \(x^I\) satisfies the condition in {Lemma~{\ref{lemma-free-monomials-of-abnormal-polynomials}}}. Therefore for every monomial \(x^I\), the word \(v(I)w\) is a Hall word, and the {abnormal polynomial} \(\abnormalpolynomial{w}{\covector}\) has the particularly simple form
\begin{equation*}
\abnormalpolynomial{w}{\covector}(x) = \sum_{I}\frac{1}{I!}\covector_{v(I)w}x^{I}\text{,}
\end{equation*}
where \(I!=i_1!\cdots i_\ell!\) for each tuple \(I=(i_1,\ldots,i_\ell)\). Since the coefficients \(\covector_{v(I)w}\) are all distinct, the required covector is defined by the monomial coefficients of the target polynomial \(P\).
\end{proof}
\begin{proposition}\label{proposition-existence-of-polynomial-factor}
Let \(\hallset\) be the deg-left-right Hall set on the free Lie algebra of rank \(r\). Let \(Q\in\polyring{\hallset}\) be a nonzero polynomial containing only variables of degrees at most \(m-1\). Then there exists a step \(s\geq m\) and a covector \(\covector\in\freelie{r,s}^*\) such that \(Q\) is a common factor of all {abnormal polynomials} \(\abnormalpolynomial{X}{\covector}\) for \(X\in\freelielayer{r,s}{m}\), and at least one of these abnormal polynomials is nonzero. Moreover, the step \(s\) only depends on the rank \(r\), the integer \(m\), and the degree of \(Q\), but not on the polynomial \(Q\) itself.
\end{proposition}
\begin{proof}\label{g:proof:idp24}
Fix a step \(s\geq m\), which will be specified more precisely later. Consider the Carnot group of step \(s\) whose Lie algebra is the quotient \(\mathfrak{g} = \freelie{r,s}/[\lowercentralseriesterm{\freelie{r,s}}{m},\lowercentralseriesterm{\freelie{r,s}}{m}]\) eliminating variables of degree \(m\) and higher. Then by {Lemma~{\ref{lemma-quotient-by-higher-variables}}}, all abnormal polynomials \(\abnormalpolynomial{X}{\covector}\) for \(X\in\lielayer{\mathfrak{g}}{m}\) in \(G\) only depend on the monomials with terms \(x_w\) of degree at most \(\deg(w)\leq m\). Moreover, by {Lemma~{\ref{lemma-lifting-quotient-abnormal-polynomials}}}, if the claim of the proposition holds in the quotient Carnot group \(G\), then the claim also holds in the free Carnot group of rank \(r\) and step \(s\).

Denote \(d:=\dim\lielayer{\mathfrak{g}}{m}\) and let \(w_{1}\lt\ldots\lt w_{d}\) be all the {deg-left-right Hall words} of degree \(m\). Denote \(q:=\deg(Q)\) and define for each Hall word \(w_{i}\) a generic polynomial
\begin{equation*}
\factorpolynomial{\factorcoefficient{i}} := \sum_{\abs{I}\leq s-m-q}\factorcoefficient{i,I}x^I
\end{equation*}
of (weighted) degree \(s-m-q\) with indeterminate coefficients \(\factorcoefficient{i,I}\). Consider the homogeneous linear system
\begin{equation*}
\abnormalpolynomial{w_{i}}{\covector} = \factorpolynomial{\factorcoefficient{i}}Q,\quad i=1,\ldots,d
\end{equation*}
in the variables \(\covector,\factorcoefficient{}\).

By {Lemma~{\ref{lemma-free-abnormal-polynomial}}}, the abnormal polynomial \(\abnormalpolynomial{w_{1}}{\covector}\) for the minimal deg-left-right Hall word \(w_{1}=1^{m-1}2\) can be freely chosen. This defines some of the unknowns \(\covector_{w}\) in terms of the indeterminates \(\factorcoefficient{1,I}\). Adding in an arbitrary substitution of the rest of the \(\covector_{w}\) variables from other equations of the system gives a substitution \(\covector = \covector(\factorcoefficient{})\) such that at least the identity \(\abnormalpolynomial{w_{1}}{\covector(\factorcoefficient{})} = \factorpolynomial{\factorcoefficient{1}}Q\) becomes trivial. Hence the substitution leads to a reduced system in the variables \(\factorcoefficient{}\) and the equations
\begin{equation}
\abnormalpolynomial{w_{i}}{\covector(\factorcoefficient{})} = \factorpolynomial{\factorcoefficient{i}}Q,\quad i=2,\ldots,d\text{.}\label{eq-abnormal-factor-system}
\end{equation}

Since the system is homogeneous, it always has the zero solution. The existence of a nontrivial solution in a large enough step \(s\) will be proved by counting the number equations and variables. For each \(i=2,\ldots,d\), every monomial \(x^I\) appearing in the polynomial \(\abnormalpolynomial{w_{i}}{\covector(\factorcoefficient{})} - \factorpolynomial{\factorcoefficient{i}}Q\) contributes one constraint on the variables \(\factorcoefficient{}\). Hence the number of equations is at most \(d-1\) times the number of monomials of degree up to \(\deg(\abnormalpolynomial{w_{i}}{\covector(\factorcoefficient{})} - \factorpolynomial{\factorcoefficient{i}}Q) \leq s-m\). The number of variables is instead \(d\) times the number of monomials of degree up to \(\deg(\factorpolynomial{\factorcoefficient{i}}) = s-m-q\).

By {Lemma~{\ref{lemma-poincare-series-for-monomials-of-bounded-degree}}}, the Poincaré series for the number of monomials in the variables of degrees \(1,\ldots,m-1\) is
\begin{equation*}
\poincare(t) = 1/\prod_{k=1}^{m-1}(1-t^k)^{\dim\freelielayer{r}{k}} =: \sum_{k=0}N_kt^k\text{.}
\end{equation*}
In terms of the series coefficients \(N_k\), the number of equations is at most
\begin{equation*}
E_s := (d-1)\sum_{k=0}^{s-m}N_k = \sum_{k=m}^s(d-1)N_{k-m}
\end{equation*}
and the number of variables is
\begin{equation*}
V_s := d\sum_{k=0}^{s-m-q}N_k = \sum_{k=m+q}^sdN_{k-m-q}\text{.}
\end{equation*}
A telescoping argument gives the identities
\begin{equation*}
(d-1)t^m\poincare(t) = \sum_{k=m}^\infty (d-1)N_{k-m}t^{k} = \sum_{k=0}^\infty (E_{k}-E_{k-1})t^{k}
\end{equation*}
and
\begin{equation*}
dt^{m+q}\poincare(t) = \sum_{k=m+q}^\infty dN_{k-m-q}t^{k} = \sum_{k=0}^\infty (V_{k}-V_{k-1})t^{k}.
\end{equation*}
That is, the difference \(V_s-E_s\) is the partial sum \(\sum_{k=0}^s\differenceinteger{k}\) of the coefficients of the series
\begin{align*}
\Delta(t) = \sum_{k=0}^\infty \differenceinteger{k}t^k \amp := dt^{m+q}\poincare(t)-(d-1)t^m\poincare(t)\\
\amp = t^m\poincare(t) - dt^m(1-t^q)\poincare(t)\text{.}
\end{align*}

The rational function \(\poincare(t)\) has a pole at \(t=1\) with \(\lim_{t\to 1-}\poincare(t)=\infty\). Since \((1-t^q)\poincare(t)\simeq (1-t)\poincare(t)\) is of higher order than \(\poincare(t)\) at \(t=1\), it follows that \(\Delta(t)\simeq \poincare(t)\) near \(t=1\). In particular
\begin{equation*}
\sum_{k=0}^\infty \differenceinteger{k}=\lim_{t\to 1-}\Delta(t) = \infty\text{.}
\end{equation*}
That is, the partial sums \(V_s-E_s = \sum_{k=0}^s \differenceinteger{k}\) tend to infinity, so there exists some large enough \(s\in \NN\) for which the number of variables \(V_s\) is strictly bigger than the upper bound \(E_s\) for the number of equations.

For such a step \(s\), the linear system {({\ref{eq-abnormal-factor-system}})} is underdetermined and has a nontrivial solution \(\overline{\factorcoefficient{}}\neq 0\), defining a solution covector \(\covector(\overline{\factorcoefficient{}})\in\mathfrak{g}^*\). Since \(\overline{\factorcoefficient{}}\neq 0\), at least one of the polynomials \(\factorpolynomial{\overline{\factorcoefficient{i}}}\) is nonzero. By assumption \(Q\) is nonzero, so it follows that at least one abnormal polynomial \(\abnormalpolynomial{w_{i}}{\covector(\overline{\factorcoefficient{}})} = \factorpolynomial{\overline{\factorcoefficient{i}}}Q\) is nonzero.

The final claim on the dependence of \(s\) follows from the construction of the coefficients \(\differenceinteger{k}\) through the generating function \(\Delta\). The rational function \(\Delta\) was determined by the integers \(r\), \(m\), \(q=\deg(Q)\), and the dimensions of the layers \(\freelielayer{r}{1},\ldots,\freelielayer{r}{m}\) of the free Lie algebra of rank \(r\). The dimensions of the layers are completely determined by \(r\), so the claim follows.
\end{proof}

\typeout{************************************************}
\typeout{Section 4 Proofs of the main theorems}
\typeout{************************************************}

\section{Proofs of the main theorems}\label{section-main-proofs}

\typeout{************************************************}
\typeout{Subsection 4.1 Abnormality of ODE trajectories}
\typeout{************************************************}

\subsection{Abnormality of ODE trajectories}\label{subsection-conclusion-of-the-proof-of-abnormality}
In {Theorem~{\ref{theorem-every-ode-trajectory-is-abnormal}}} the Carnot group \(F\) is only concerned with the degree of the polynomial ODE and not the particular ODE itself. Making this dependency explicit leads to the following stronger statement.
\begin{theorem}\label{theorem-every-ode-trajectory-is-abnormal-quantified}
For every \(r\geq 2\) and \(d\in\NN\) there exists \(s\in\NN\) such that the following holds:

Let \(P\colon G\to TG\) be a polynomial vector field in a Carnot group \(G\) of rank \(r\) and let \(P=\sum_{i=1}^nP_i\tilde{X}_i\) be its expansion in a left-invariant frame \(\tilde{X}_i\) of vector fields adapted to the stratification on \(G\). Suppose that the polynomials \(P_i\colon G\to\RR\) for the horizontal components \(i=1,\ldots,r\) have bounded degree \(\deg(P_i)\leq d\). Then all the horizontal trajectories of the ODE \(\dot{x} = P(x)\) {lift} to abnormal curves in the free Carnot group \(\freecarnot{r,s}\) of rank \(r\) and step \(s\).
\end{theorem}
\begin{proof}\label{g:proof:idp25}
By the expansion of the vector field \(P\) in the left-invariant frame, if \(x\colon(a,b)\to G\) is a horizontal trajectory of the vector field \(P\), then the derivative of the trajectory is
\begin{equation}
\frac{d}{dt}x(t) = P(x(t)) = \sum_{i=1}^r P_i(x(t))\tilde{X}_i(x(t))\text{.}\label{eq-horizontal-trajectory}
\end{equation}

Since the statement of the theorem only concerns horizontal trajectories there is no loss of generality in assuming that the other coefficients \(P_{r+1},\ldots,P_n\) are all zero. Then the vector field \(P\) has a well defined lift from \(G\) to a horizontal vector field \(\tilde{P}\) in the free Carnot group of rank \(r\) and step equal to the step of \(G\), and all horizontal trajectories of \(P\) lift to trajectories of the vector field \(\tilde{P}\). Hence the claim holds if and only if the claim holds when \(G\) is a free Carnot group. In particular, there is no loss of generality in assuming that \(G\) is a Carnot group that is {compatible} with the {deg-left-right Hall set} \(\hallset\). By considering coordinates {adapted} to the Hall set \(\hallset\), the polynomials \(P_1,\ldots,P_r\) are identified with elements of {\(\polyring{\hallset}\)}.

Let \(Q_{1},\ldots,Q_{r}\in\polyring{\hallset}\) be polynomials such that not all \(Q_{i}\) are zero, \(\sum P_iQ_{i} = 0\), and \(\max_i\deg(Q_{i})\leq \max_i\deg(P_i)\leq d\). By {Lemma~{\ref{lemma-vector-of-polynomials-as-derivatives-in-free-lie-algebra}}}, there exists a polynomial \(Q\in\polyring{\hallset}\) such that \(X_iQ=Q_{i}\). Let \(x(t)\) be a horizontal trajectory of \(P\) whose closure contains the identity. Since \(\sum P_i(X_iQ)=0\), the identity {({\ref{eq-horizontal-trajectory}})} implies that the polynomial \(Q\) is constant along the horizontal trajectory \(x(t)\). Adding in a constant if necessary, there is no loss of generality in assuming that \(Q(x(t))\equiv 0\).

From the construction of {Lemma~{\ref{lemma-vector-of-polynomials-as-derivatives-in-free-lie-algebra}}}, it follows that the polynomial \(Q\) only contains variables of degrees at most \(\max_i\deg(Q_{i})+1\leq d+1\) and the degree \(\deg(Q)\) is bounded by the same quantity. Let \(m:=d+2\). By {Proposition~{\ref{proposition-existence-of-polynomial-factor}}}, there exists some high enough step \(s\geq m\) and a covector \(\covector\in\freelie{r,s}^*\) such that the polynomial \(Q\) is a common factor for all abnormal polynomials \(\abnormalpolynomial{X}{\covector}\) for \(X\in\freelielayer{r,s}{m}\). Since the closure of the trajectory \(x(t)\) by assumption contains the identity, {Lemma~{\ref{lemma-curves-annihilating-higher-order-abnormal-polynomials}}} implies that the trajectory \(x(t)\) has an abnormal lift in the free Carnot group \(\freecarnot{r,s}\).

For an arbitrary horizontal trajectory \(x(t)\) of the vector field \(P\) in \(G\), let \(z\in G\) be any point in the closure of the trajectory, and let \(y(t):=L_{z^-1}x(t)\) be the left translation of the trajectory \(x(t)\). The derivative of the translated trajectory is
\begin{equation*}
\frac{d}{dt}y(t) = \sum_{i=1}^r P_i(x(t))\Big((L_{z^{-1}})_*\tilde{X}_i(x(t))\Big)
= \sum_{i=1}^r P_i(L_{z}y(t))\tilde{X}_i(y(t))\text{.}
\end{equation*}

Since left translations preserve the degrees of polynomials, the earlier argument implies that there exists a polynomial \(Q_z\in\polyring{\hallset}\) such that \(Q_z(y(t))\equiv 0\) and the degree of \(Q_z\) and maximum degree of variables of \(Q_z\) is again strictly less than \(m\). Repeating the rest of the argument, {Proposition~{\ref{proposition-existence-of-polynomial-factor}}} implies that the trajectory \(y(t)\) also has an abnormal lift in the free Carnot group \(\freecarnot{r,s}\) for the same step \(s\geq m\) as before. Since abnormality is preserved under left translation, it follows that the original horizontal trajectory \(x(t)\) of \(P\colon \freecarnot{r,s}\to T\freecarnot{r,s}\) lifts to an abnormal curve in \(F\), proving the claim.
\end{proof}
An explicit value for the abnormality step \(s\) of {Theorem~{\ref{theorem-every-ode-trajectory-is-abnormal-quantified}}} can be computed from the rank \(r\) and degree bound \(d\) as described in {Algorithm~{\ref{algorithm-abnormal-trajectory}}}. The resulting values of \(s\) for some of the smallest values of \(r\) and \(d\) are collected in {Figure~{\ref{table-nilpotency-step-bounds}}}.

\typeout{************************************************}
\typeout{Subsection 4.2 Concatenation of abnormals}
\typeout{************************************************}

\subsection{Concatenation of abnormals}\label{subsection-concatenation-of-abnormals}
\begin{proof}[Proof of Theorem~{\ref{theorem-concatenation-of-abnormals}}.]\label{g:proof:idp26}
Let \(r\) be the rank of \(G\) and let \(s\) be the nilpotency step of \(G\). Since the abnormal curves in \(G\) lift to the free Carnot group \(\freecarnot{r,s}\) of rank \(r\) and step \(s\), replacing \(G\) with \(\freecarnot{r,s}\) if necessary, there is no loss of generality in assuming that \(G\) is {compatible} with the {deg-left-right Hall set} \(\hallset\).

\label{g:notation:idp27} Let \(\alpha\colon[a,b]\to G\) and \(\beta\colon[c,d]\to G\) be two abnormal curves. The concatenation \(\alpha\star\beta\colon[a,b-c+d]\to G\) in the group \(G\) is defined by translating \(\beta(c)\) to \(\alpha(b)\) and then concatenating curves in the usual way. Since abnormality is preserved by left translation and reparametrization, it suffices to consider the case \([a,b]=[-1,0]\), \([c,d]=[0,1]\), and \(\alpha(0)=\beta(0)=\identity{G}\). Then the concatenation \(\alpha\star\beta\colon[-1,1]\to G\) is defined by
\begin{equation*}
\alpha\star\beta(t) = \begin{cases}\alpha(t),\amp t\leq 0\\\beta(t),\amp t\geq 0,\end{cases}\text{.}
\end{equation*}

By the characterization of {Lemma~{\ref{lemma-curves-in-abnormal-variety-are-abnormal}}}, abnormality means that there exist nonzero covectors \(\covector_\alpha,\covector_\beta\in\mathfrak{g}^*\) such that \(\abnormalpolynomial{X}{\covector_\alpha}\circ\alpha\equiv 0\) and \(\abnormalpolynomial{X}{\covector_\beta}\circ\beta\equiv 0\) for all horizontal vectors \(X\in\lielayer{\mathfrak{g}}{1}\).

Let \(X,Y\in\lielayer{\mathfrak{g}}{1}\) be such that the abnormal polynomials \(\abnormalpolynomial{X}{\covector_\alpha}\) and \(\abnormalpolynomial{Y}{\covector_\beta}\) are nonzero, and let \(Q:= \abnormalpolynomial{X}{\covector_\alpha}\abnormalpolynomial{Y}{\covector_\beta}\) be the product of the two. By {Proposition~{\ref{proposition-existence-of-polynomial-factor}}} there exists a step \(s_2\geq s\) and a covector \(\covectorb\in\freelie{r,s_2}^*\) such that \(Q\) is a common factor of all the abnormal polynomials \(\abnormalpolynomial{Z}{\covectorb}\) for \(Z\in\freelielayer{r,s_2}{s}\). By construction \(Q\circ(\alpha\star\beta)\equiv 0\), so {Lemma~{\ref{lemma-curves-annihilating-higher-order-abnormal-polynomials}}} then implies that the concatenation \(\alpha\star\beta\) has an abnormal lift in the free Carnot group \(F=\freecarnot{r,s_2}\).

By {Proposition~{\ref{proposition-existence-of-polynomial-factor}}} the step \(s_2\) depends only on the variables and degree of \(Q\in\polyring{\hallset}\). Since the step of the Carnot group \(G\) is \(s\), the polynomials \(\abnormalpolynomial{X}{\covector_\alpha}\) and \(\abnormalpolynomial{X}{\covector_\beta}\) have (weighted) degree at most \(s-1\) and hence \(Q\) has degree at most \(2s-2\) regardless of the specific covectors \(\covector_\alpha\) and \(\covector_\beta\). Since the variables are necessarily some subset of the variables of the group \(G\), it follows that the step \(s_2\) has an upper bound that depends only on the group \(G\), and not on the individual curves \(\alpha\) and \(\beta\).
\end{proof}

\typeout{************************************************}
\typeout{Section 5 Examples}
\typeout{************************************************}

\section{Examples}\label{section-examples}
The proof of {Theorem~{\ref{theorem-every-ode-trajectory-is-abnormal-quantified}}} of abnormality of ODE trajectories gives a constructive method to find any polynomial ODE trajectory as an abnormal curve. This section covers the method more concretely, showing how the correct Carnot group and covector is found for some specific examples. The practical version of the proof of {Theorem~{\ref{theorem-every-ode-trajectory-is-abnormal-quantified}}} is the following algorithm.
\begin{algorithm}\label{algorithm-abnormal-trajectory}
Input: a polynomial ODE \(\dot{x}=P(x)\) as a vector of polynomials \((P_1,\ldots,P_r)\). Output: a step \(s\) and a covector \(\covector\) such that trajectories of the ODE through the identity are abnormal with covector \(\covector\) in the free Carnot group \(\freecarnot{r,s}\) of rank \(r\) and step \(s\).
\begin{enumerate}[label=\arabic*.]
\item\label{algorithm-1-orthogonal-poly}Choose a nonzero vector of polynomials \((Q_{1},\ldots,Q_{r})\) orthogonal to \((P_1,\ldots,P_r)\). Let \(Q\in\polyring{\hallset}\) be an (abstract) polynomial such that \(X_iQ=Q_{i}\).
\item\label{algorithm-2-nonzero-derivatives}Compute all nonzero higher order derivatives \(X_{w_{i}}Q\) for {deg-left-right Hall words} \(w_{i}\) of degree \(\geq 2\).
\item\label{algorithm-3-pde-integration}Solve the resulting PDE for the polynomial \(Q\) by sequentially integrating in each variable \(x_{w_{i}}\) in decreasing deg-left-right Hall order. Let \(m\geq 2\) be such that \(Q\) only contains variables \(x_{w}\) of degrees at most \(m-1\).
\item\label{algorithm-4-sufficient-step}Find a step \(s\geq m\) such that \(\sum_{k=0}^{s}\differenceinteger{k} \geq 1\), where the integers \(\differenceinteger{k}\) are determined by the generating function
\begin{equation*}
\Delta(t) = \sum_{k=0}^\infty \differenceinteger{k}t^k = \frac{\Big(1-(\dim \freelielayer{r}{m})(1-t^{\deg(Q)})\Big)t^m}{\prod_{k=1}^{m-1}(1-t^k)^{\dim\freelielayer{r}{k}}}\text{.}
\end{equation*}

\item\label{algorithm-5-abnormal-polys}Compute the polynomials \(R_i := \abnormalpolynomial{w_{i}}{\covector}-\factorpolynomial{\factorcoefficient{i}}Q\), where
\begin{itemize}[label=\textbullet]
\item{}\(w_{1},\ldots,w_{\dim \freelielayer{r}{m}}\) are all the deg-left-right Hall words of degree \(m\),
\item{}\(\covector\in\mathfrak{g}^*\) is a covector in the dual of the quotient \(\mathfrak{g} = \freelie{r,s}/[\lowercentralseriesterm{\freelie{r,s}}{m},\lowercentralseriesterm{\freelie{r,s}}{m}]\) with indeterminate coefficients \(\covector_{w}\), and
\item{}\(\factorpolynomial{\factorcoefficient{i}}\) are generic polynomials of degree \(s-m-\deg(Q)\) whose coefficients are indeterminates \(\factorcoefficient{i,I}.\)
\end{itemize}

\item\label{algorithm-6-linear-system}Solve the linear system \(R_i=0\), \(i=1,\ldots,\dim \freelielayer{r}{m}\) in the variables \(\covector,\factorcoefficient{}\). The \(\covector\) component of any solution  is a covector for which trajectories reaching zero of the ODE \(\dot{x}=P(x)\) lift to abnormals in the free Carnot group \(\freecarnot{r,s}\).
\end{enumerate}

\end{algorithm}
An implementation of {Algorithm~{\ref{algorithm-abnormal-trajectory}}} using the SageMath computer algebra system \cite{sagemath-9-0} is available in \cite{software-ode_abnormals}.
\begin{remark}\label{remark-algorithm-nonsharp}
The naive upper bound computed in step \ref{algorithm-4-sufficient-step} of {Algorithm~{\ref{algorithm-abnormal-trajectory}}} for the nilpotency step \(s\) as a function of the rank and degree of the polynomials is in general horribly inefficient, see the explosive growth already for the smallest ranks \(r\) and degrees \(d\) listed in {Figure~{\ref{table-nilpotency-step-bounds}}}. In practice, solutions are found in much smaller nilpotency steps, see {Subsection~{\ref{example-lorenz-butterfly}}} for a particularly egregious example where the a priori upper bound is \(s=724\), but the system has nontrivial solutions already for \(s=13\). For practical computations it is more efficient to form the linear system in smaller steps \(s'\leq s\) and increase \(s'\) one by one until a nontrivial solution is found.

In fact, computations suggest that there exists a much more refined bound \(s=s(d)\) independent from the rank \(r\). Linear ODEs with randomly chosen coefficients in ranks 2, 3, and 4 always had solutions in step \(s=7\), despite the a priori bounds for \(s\) being 11, 89, and 386 respectively. Similarly the quadratic ODEs of {Subsection~{\ref{example-the-hawaiian-earring}}} and {Subsection~{\ref{example-lorenz-butterfly}}} in ranks 2 and 3 both have solutions in step \(s=13\) despite the increase of the a priori bound from \(s=38\) to \(s=724\).
\end{remark}
\begin{figure}[hbtp]%
\centering%
{%
\begin{tabular}{r|rrrrr}%
\(r\)\textbackslash{}\(d\)&1&2&3&4&5\tabularnewline\hline%
2&11&38&172&577&2372\tabularnewline[0pt]%
3&89&724&6034&46036&365813\tabularnewline[0pt]%
4&386&5322&73109&983505&13529000%
\end{tabular}%
}%
\caption{Some values for \(s=s(r,d)\) of {Theorem~{\ref{theorem-every-ode-trajectory-is-abnormal-quantified}}}.}%
\label{table-nilpotency-step-bounds}%
\end{figure}

\typeout{************************************************}
\typeout{Subsection 5.1 A logarithmic spiral (r2s7)}
\typeout{************************************************}

\subsection{A logarithmic spiral (r2s7)}\label{example-rank-2-spirals}
\begin{figure}[hbtp]%
\centering%
\resizebox{0.3\linewidth}{!}{\begin{tikzpicture}%
\def\rmax{8}%
\def\step{0.5}%
\foreach \i in {0,\step,...,\rmax}%
{%
\draw[line width={0.25-0.25*\i/\rmax}, domain={\i*pi}:{(\i+\step)*pi}, smooth] plot ({exp(-0.25*\x)*cos(\x r)},{exp(-0.25*\x)*sin(\x r)});%
}%
\end{tikzpicture}%
}%
\caption{The logarithmic spiral \(e^{-t/4}(\cos t,\sin t)\)}%
\label{figure-logarithmic-spiral}%
\end{figure}
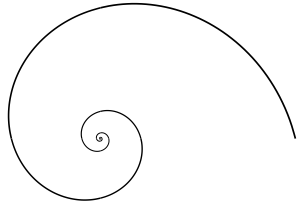
The logarithmic spiral
\begin{equation*}
\gamma\colon [0,\infty)\to\RR^2,\quad \gamma(t) = e^{-t}(\cos t,\sin t)
\end{equation*}
is a trajectory of the linear ODE
\begin{align*}
\dot{x}_1 \amp= P_1(x) = -x_1-x_2\\
\dot{x}_2 \amp= P_2(x) = x_1-x_2
\end{align*}
with \(\lim\limits_{t\to\infty}\gamma(t)=0\). The construction of {Algorithm~{\ref{algorithm-abnormal-trajectory}}} will show that the logarithmic spiral \(\gamma\) lifts to an abnormal curve in the free Carnot group of rank 2 and step 7.

Step \ref{algorithm-1-orthogonal-poly} Define polynomials \(Q_{1}:=x_1-x_2\) and \(Q_{2}:=x_1+x_2\) so that \((Q_{1},Q_{2})\) is orthogonal to the ODE vector \((P_1,P_2)\).

Step \ref{algorithm-2-nonzero-derivatives} Suppose \(Q\in\polyring{\hallset}\) is a polynomial such that \(X_1Q = Q_{1}\) and \(X_2Q = Q_{2}\) and compute the higher order derivatives. The restricted {action} \(\freelie{2}\acts\polyring{x_1,x_2}\) is defined by \(X_1=\partial_1\) and \(X_2=\partial_2\), as can be seen by considering the coordinate expressions of the left-invariant horizontal vector fields in any rank 2 Carnot group. Since \(\deg(X_1Q)=\deg(X_2Q)=1\), the only nontrivial derivative is
\begin{equation*}
X_{12}Q = [X_1,X_2]Q = X_1(X_2Q)-X_2(X_1Q) = 2\text{,}
\end{equation*}
and all the higher order derivatives are zero.

Step \ref{algorithm-3-pde-integration} In the {action} \(\freelie{2}\acts\polyring{\hallset}\), each Hall basis element acts by \(X_{w}=\partial_{w}+\rho_{w}\), with \(\rho_{w}\) some derivation whose kernel contains all polynomials in the variables of degree equal or lower than \(w\). Hence the nonzero partial derivatives determine the variables of the polynomial \(Q\), which in this case means that \(Q\in\polyring{x_1,x_2,x_{12}}\). The action \(\freelie{2}\acts\polyring{x_1,x_2,x_{12}}\) is then determined by exponential coordinates {adapted} to the {deg-left-right Hall set} on any rank 2 Carnot group of step at least \(2\), the prototypical example being the Heisenberg group. Explicitly, the derivations are
\begin{align*}
X_1 \amp= \partial_1\amp
X_2 \amp= \partial_2+x_1\partial_{12}\amp
X_{12} \amp= \partial_{12}\text{.}
\end{align*}

Integrating in the maximal variable \(x_{12}\) gives
\begin{equation*}
Q = 2x_{12} + Q^{(2)}
\end{equation*}
for some polynomial \(Q^{(2)}\in\polyring{x_1,x_2}\). The remainder \(Q^{(2)}\) satisfies the PDE
\begin{align*}
\partial_1Q^{(2)} \amp= X_1(Q-2x_{12}) = X_1Q = x_1-x_2\\
\partial_2Q^{(2)} \amp= X_2(Q-2x_{12}) = X_2Q - 2x_1 = -x_1+x_2\text{.}
\end{align*}
Integrating in the variables \(x_2,x_1\), a solution is \(Q^{(2)} = \frac{1}{2}x_1^2-x_1x_2+\frac{1}{2}x_2^2\), so the full solution is
\begin{equation}
Q = \frac{1}{2}x_1^2-x_1x_2+\frac{1}{2}x_2^2 + 2x_{12}\text{.}\label{eq-Q}
\end{equation}

Step \ref{algorithm-4-sufficient-step} The polynomial \(Q\) contains variables of degrees \(1\) and \(2\), so \(m=3\). The dimensions of the layers \(1,2,3\) of the free Lie algebra of rank 2 are \(2,1,2\), respectively. To determine a sufficient nilpotency step, the generating function to consider is
\begin{equation*}
\Delta(t) = \frac{(1-2(1-t^2))t^3}{(1-t)^2(1-t^2)}\text{.}
\end{equation*}
The first few terms of its series expansion are
\begin{equation*}
\Delta(t) = - t^{3} -2 t^{4} -2 t^{5} -2 t^{6} - t^{7} + 2 t^{9} + 4 t^{10} + 7 t^{11} + \cdots\text{.}
\end{equation*}
Since \(-1-2-2-2-1+2+4+7 = 5\geq 1\), the logarithmic spiral will have an abnormal lift at least in step \(11\).

Step \ref{algorithm-5-abnormal-polys} Let \(\mathfrak{g}=\freelie{2,11}/[\lowercentralseriesterm{\freelie{2,11}}{3},\lowercentralseriesterm{\freelie{2,11}}{3}]\) be the quotient eliminating variables of degree \(\geq 3\) from the {abnormal polynomials} \(\abnormalpolynomial{112}{\covector}\) and \(\abnormalpolynomial{212}{\covector}\), which in nilpotency step 11 are polynomials of degree 8. By {Lemma~{\ref{lemma-free-abnormal-polynomial}}}, the abnormal polynomial \(\abnormalpolynomial{112}{\covector}\) has the simple expression
\begin{equation}
\abnormalpolynomial{112}{\covector}(x_1,x_2,x_{12}) = \sum_{a+b+2c\leq 8}\frac{\covector_{(12)^c(2)^b(1)^{a+2}2}}{a!b!c!}x_1^ax_2^bx_{12}^c\text{.}\label{eq-P112}
\end{equation}

The abnormal polynomial \(\abnormalpolynomial{212}{\covector}\) is not as simple, since the Lie bracket \(\ad{X_{12}}^c\ad{X_2}^b\ad{X_1}^aX_{212}\) is a {deg-left-right Hall tree} only when \(a=0\) by the characterization of {Lemma~{\ref{lemma-hall-words-for-deg-left-right-order}}}. For \(a\gt 0\), a normal form may be computed by considering the restricted adjoint representation \(\ad{}\colon\mathfrak{g}\to\mathfrak{gl}(\lowercentralseriesterm{\mathfrak{g}}{3})\). By the construction of the quotient \(\mathfrak{g}\), the only nontrivial commutator in \(\mathfrak{gl}(\lowercentralseriesterm{\mathfrak{g}}{3})\) is \([\ad{X_1},\ad{X_2}] = \ad{X_{12}}\). That is, a normal form may be computed using only the relation
\begin{equation*}
\ad{X_1}\ad{X_2} = \ad{X_2}\ad{X_1} + \ad{X_{12}}\text{.}
\end{equation*}
Applying the above \(a\) times, the resulting normal form for \(a\gt 0\) is
\begin{align*}
\ad{X_{12}}^c\ad{X_2}^b\ad{X_1}^aX_{212} \amp= \ad{X_{12}}^c\ad{X_2}^b\ad{X_1}^a\ad{X_2}X_{12}\\
\amp= \ad{X_{12}}^c\ad{X_2}^{b+1}\ad{X_1}^aX_{12} + (a-1)\ad{X_{12}}^{c+1}\ad{X_2}^{b}\ad{X_1}^{a-1}X_{12}\text{.}
\end{align*}
{Lemma~{\ref{lemma-abnormal-polynomials-in-exponential-coordinates}}} then gives the explicit expression
\begin{align}
\abnormalpolynomial{212}{\covector}\amp(x_1,x_2,x_{12})
= \sum_{b+2c\leq 8}\frac{\covector_{(12)^c(2)^{b+1}12}}{b!c!}x_2^bx_{12}^c\label{eq-P212}\\
\amp+ \sum_{\substack{a+b+2c\leq 8\\a\geq 1}}\frac{\covector_{(12)^c(2)^{b+1}(1)^{a+1}2}+(a-1)\covector_{(12)^{c+1}(2)^b(1)^a2}}{a!b!c!}x_1^ax_2^bx_{12}^c\text{.}\notag
\end{align}

For \(i=1,2\), define polynomials
\begin{equation}
\factorpolynomial{\factorcoefficient{i}}(x_1,x_2,x_{12}) := \sum_{a+b+2c\leq 6}\factorcoefficient{i,a,b,c}x_1^ax_2^bx_{12}^c\label{eq-factorpoly}
\end{equation}
and compute for \(w_{1}=112\) and \(w_{2}=212\) the difference polynomials
\begin{equation}
R_i = \sum_{a+b+2c\leq 8}R_{i,a,b,c}(\covector,\factorcoefficient{i})x_1^ax_2^bx_{12}^c := \abnormalpolynomial{w_{i}}{\covector}-\factorpolynomial{\factorcoefficient{i}}Q\label{eq-Ri}
\end{equation}
using the explicit expressions {({\ref{eq-Q}})}\textendash{}{({\ref{eq-factorpoly}})}.

Step \ref{algorithm-6-linear-system} Consider the linear system
\begin{equation*}
R_{i,a,b,c}(\covector,\factorcoefficient{i}) = 0,\quad i=1,2, \quad a+b+2c\leq 8\text{.}
\end{equation*}
In the full set of \((\covector,\factorcoefficient{})\) variables, there are 190 equations and 220 variables: 120 variables \(\covector_{112},\covector_{212},\ldots\) and 100 variables \(\factorcoefficient{i,a,b,c}\). Reducing to a system in only the \(\factorcoefficient{}\) variables as in the proof of {Proposition~{\ref{proposition-existence-of-polynomial-factor}}} leaves a system of 95 equations and 100 variables. The solution space is however much bigger than the difference: there is a 38 dimensional space of solutions \((\covector,\factorcoefficient{})\), all with a nonzero \(\factorcoefficient{}\) component.

One of the simplest solutions gives a degree 7 covector
\begin{align*}
\covector\amp= 3 \covector_{1111112}  - 3 \covector_{2211112}  + 6 \covector_{2221112}  - 9 \covector_{2222112}  + 18 \covector_{2222212} \\
\amp + 2 \covector_{1211112}  - 2 \covector_{1222112}  + 8 \covector_{1222212}  + 4 \covector_{1212112}  + 8 \covector_{1212212}
\end{align*}
Substituting the above solution \(\covector\) into the formulas {({\ref{eq-P112}})} and {({\ref{eq-P212}})} and factoring gives the abnormal polynomials
\begin{align*}
\abnormalpolynomial{112}{\covector} \amp = \frac{1}{4} (x_{1}^{2} + 2 x_{1} x_{2} - 3 x_{2}^{2} + 4 x_{12})Q\\
\abnormalpolynomial{212}{\covector} \amp = \frac{1}{2}(x_{1}^{2} + 3 x_{2}^{2} + 4 x_{12})Q
\end{align*}
The conclusion is that the logarithmic spiral \(\gamma(t)=e^{-t}(\cos t,\sin t)\) lifts to an abnormal curve in the free Carnot group of rank 2 and step 7.

\typeout{************************************************}
\typeout{Subsection 5.2 Planar linear ODEs (r2s7)}
\typeout{************************************************}

\subsection{Planar linear ODEs (r2s7)}\label{example-parametric-families-of-odes}
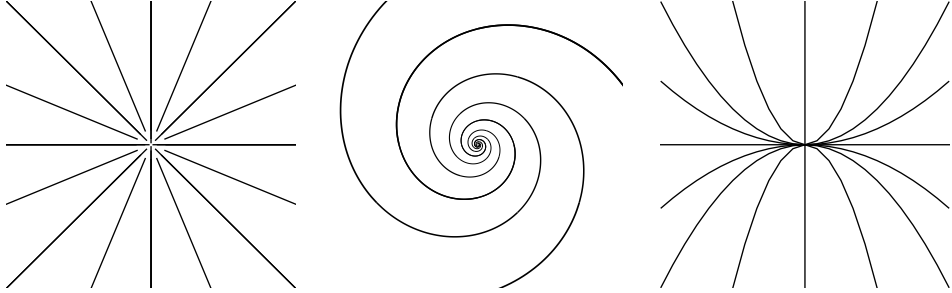
\begin{figure}[hbtp]%
\centering%
\resizebox{0.3\linewidth}{!}{\begin{tikzpicture}[scale=1.5]%
\clip (-1,-1) rectangle(1,1);%
\foreach \offset/\cutoff/\lines in {0/0.01/2, 45/0.04/4, 22.5/0.1/8}%
{%
\foreach \i in {1,...,\lines}%
{%
\draw ({180+180*\i/\lines}:2) -- ({180+180*\i/\lines}:\cutoff) ({180*\i/\lines}:\cutoff)-- ({180*\i/\lines}:2);%
}%
}%
\end{tikzpicture}}\hspace{0.04\linewidth}%
\resizebox{0.3\linewidth}{!}{\begin{tikzpicture}[scale=1.5]%
\clip (-1,-1) rectangle(1,1);%
\def\rmax{6}%
\def\step{0.5}%
\foreach \i in {0,\step,...,\rmax}%
{%
\def\scale{1.2}%
\foreach \j in {0,...,3}%
{%
\draw[line width={0.5-0.5*\i/\rmax}, domain={\i*pi}:{(\i+\step)*pi}, smooth] plot ({\scale*exp(-0.25*\x)*cos((\x + 2*pi/3*\j) r)},{\scale*exp(-0.25*\x)*sin((\x + 2*pi/3*\j) r)});%
}%
}%
\end{tikzpicture}}\hspace{0.04\linewidth}%
\resizebox{0.3\linewidth}{!}{\begin{tikzpicture}[scale=1.5]%
\clip (-1,-1) rectangle(1,1);%
\draw (-2,0)--(2,0);%
\draw (0,-2,0)--(0,2);%
\def\curves{3}%
\foreach \i in {1,...,\curves}%
{%
\draw[domain=-1:1] plot({\x},{(\curves-1)*(\curves-1)*\x*\x/(\i*\i)})%
plot({\x},{-(\curves-1)*(\curves-1)*\x*\x/(\i*\i)});%
}%
\end{tikzpicture}}%
\caption{Some trajectories of planar linear ODEs}%
\label{figure-planar-phase-portraits}%
\end{figure}
The technique of {Algorithm~{\ref{algorithm-abnormal-trajectory}}} can be applied to ODEs depending on free parameters \(\freeparam_1,\ldots,\freeparam_k\in\RR\) by replacing all the considerations over the field \(\RR\) with the polynomial ring \(\polyring{\freeparam_1,\ldots,\freeparam_k}\) or the fraction field \(\fractionfield{\freeparam_1,\ldots,\freeparam_k}\) where necessary. Solving the linear system in step \ref{algorithm-6-linear-system} is where the only difference appears. The difference is that solving a linear system with coefficients in \(\fractionfield{\freeparam_1,\ldots,\freeparam_k}\) does not yield a universal solution, since the resulting nonzero covector may vanish for specific choices of parameters. Accounting for the vanishing leads to a semialgebraic description of the covector with finitely many different expressions depending on the parameters \(\freeparam_1,\ldots,\freeparam_k\).

Consider a generic homogeneous planar linear ODE
\begin{align}
\dot{x}_1 \amp = \freeparam_{11}x_1 + \freeparam_{12}x_2\label{eq-linear-ode}\\
\dot{x}_2 \amp = \freeparam_{21}x_1 + \freeparam_{22}x_2\notag
\end{align}
with parameters \(\freeparam_{11},\freeparam_{12},\freeparam_{21},\freeparam_{22}\in\RR\). As in {Subsection~{\ref{example-rank-2-spirals}}} for a logarithmic spiral, following {Algorithm~{\ref{algorithm-abnormal-trajectory}}} will show that for any such ODE all trajectories whose closures meet the origin lift to abnormals in the free Carnot group of rank 2 and step 7.

Step \ref{algorithm-1-orthogonal-poly} Let \(\fractionfield{\vectorparam{\freeparam}}:=\fractionfield{\freeparam_{11},\freeparam_{12},\freeparam_{21},\freeparam_{22}}\) be the fraction field with the parameters as indeterminates. Suppose \(Q\in\polyringextended{\vectorparam{\freeparam}}{\hallset}\) is a polynomial with coefficients in the fraction field such that \(X_1Q=\freeparam_{21}x_1 + \freeparam_{22}x_2\) and \(X_2Q=-\freeparam_{11}x_1 - \freeparam_{12}x_2\).

Steps \ref{algorithm-2-nonzero-derivatives}\textendash{}\ref{algorithm-5-abnormal-polys} proceed exactly as in {Subsection~{\ref{example-rank-2-spirals}}}, since the {action} \(\freelie{2}\acts\polyring{\hallset}\), the {Poincaré series}, and the {abnormal polynomials} are all essentially uneffected by the field extension \(\RR\into \fractionfield{\vectorparam{\freeparam}}\), and \(\deg(X_1Q)=\deg(X_2Q)=1\) as before. The solution of the PDE for \(Q\) is
\begin{equation*}
Q = \frac{1}{2}\freeparam_{21}x_1^2 + \freeparam_{22}x_1x_2 - \frac{1}{2}\freeparam_{12}x_2^2 -(\freeparam_{11}+\freeparam_{22})x_{12}\text{.}
\end{equation*}
The abnormal polynomials \(\abnormalpolynomial{112}{\covector},\abnormalpolynomial{212}{\covector}\) are exactly the same as in {({\ref{eq-P112}})} and {({\ref{eq-P212}})}, and the difference polynomial coefficients \(R_{i,a,b,c}(\covector,\factorcoefficient{i})\) are again defined by {({\ref{eq-Ri}})}.

Step \ref{algorithm-6-linear-system} The solution space in the Lie algebra \(\freelie{2,11}/[\lowercentralseriesterm{\freelie{2,11}}{3},\lowercentralseriesterm{\freelie{2,11}}{3}]\) of nilpotency step 11 is a 31 dimensional space over \(\fractionfield{\vectorparam{\freeparam}}\). The simplest solutions are found already in step \(s'=7\). In \(\freelie{2,7}/[\lowercentralseriesterm{\freelie{2,7}}{3},\lowercentralseriesterm{\freelie{2,7}}{3}]\), the solution space is 2 dimensional and an example solution has the nonzero coefficients
\begin{align*}
\covector_{(1)^{4}(112)} \amp= 3  \freeparam_{21}^2  (6\freeparam_{11}^2 - 5\freeparam_{12}\freeparam_{21} + 13\freeparam_{11}\freeparam_{22} + 2\freeparam_{22}^2)\\
\covector_{(2)(1)^{3}(112)} \amp= 3  \freeparam_{21}  (\freeparam_{11}\freeparam_{12}\freeparam_{21} + 5\freeparam_{11}^2\freeparam_{22} - 3\freeparam_{12}\freeparam_{21}\freeparam_{22} + 11\freeparam_{11}\freeparam_{22}^2 + 2\freeparam_{22}^3)\\
\covector_{(2)^{2}(1)^{2}(112)} \amp= -4\freeparam_{11}^2\freeparam_{12}\freeparam_{21} + 3\freeparam_{12}^2\freeparam_{21}^2 - 3\freeparam_{11}\freeparam_{12}\freeparam_{21}\freeparam_{22}\\
\amp+ 8\freeparam_{11}^2\freeparam_{22}^2 - 2\freeparam_{12}\freeparam_{21}\freeparam_{22}^2 + 18\freeparam_{11}\freeparam_{22}^3 + 4\freeparam_{22}^4\\
\covector_{(2)^{3}(1)(112)} \amp= -3  \freeparam_{12}  (\freeparam_{11}\freeparam_{12}\freeparam_{21} + 3\freeparam_{11}^2\freeparam_{22} - \freeparam_{12}\freeparam_{21}\freeparam_{22} + 5\freeparam_{11}\freeparam_{22}^2)\\
\covector_{(2)^{4}(112)} \amp= -3  \freeparam_{12}^2  (-2\freeparam_{11}^2 + \freeparam_{12}\freeparam_{21} - \freeparam_{11}\freeparam_{22} + 2\freeparam_{22}^2)\\
\covector_{(12)(1)^{2}(112)} \amp= \freeparam_{21}  (\freeparam_{11} + \freeparam_{22})  (-4\freeparam_{11}^2 + 3\freeparam_{12}\freeparam_{21} - 9\freeparam_{11}\freeparam_{22} - 2\freeparam_{22}^2)\\
\covector_{(12)(2)(1)(112)} \amp= - (\freeparam_{11} + \freeparam_{22})  (\freeparam_{11}\freeparam_{12}\freeparam_{21} + 3\freeparam_{11}^2\freeparam_{22} - \freeparam_{12}\freeparam_{21}\freeparam_{22} + 7\freeparam_{11}\freeparam_{22}^2 + 2\freeparam_{22}^3)\\
\covector_{(12)(2)^{2}(112)} \amp= \freeparam_{12}  (\freeparam_{11} + \freeparam_{22})  (2\freeparam_{11}^2 - \freeparam_{12}\freeparam_{21} + 3\freeparam_{11}\freeparam_{22})\\
\covector_{(12)^{2}(112)} \amp= -(\freeparam_{11} + \freeparam_{22})^2  (-2\freeparam_{11}^2 + \freeparam_{12}\freeparam_{21} - 5\freeparam_{11}\freeparam_{22} - 2\freeparam_{22}^2)\\
\covector_{(2)^{4}(212)} \amp= 12  (\freeparam_{11} + \freeparam_{22})  \freeparam_{12}^3\\
\covector_{(12)(2)^{2}(212)} \amp= 2  \freeparam_{12}^2  (\freeparam_{11} + \freeparam_{22})^2
\end{align*}

For some specializations \(\vectorparam{\freeparam}\in\RR^4\) such as \(\freeparam_{11}=-2\), \(\freeparam_{12}=\freeparam_{21}=0\), \(\freeparam_{21}=1\), the above covector vanishes and hence is not a valid abnormal covector. Nonetheless the existence of a generic solution implies that solutions exist also for every other choice of the parameters by the following brief argument:
\begin{lemma}\label{lemma-kernel-of-specialization}
If a matrix with coefficients in a polynomial ring \(\polyring{a_1,\ldots,\freeparam_k}\) has a nontrivial kernel over the fraction field \(\fractionfield{\freeparam_1,\ldots,\freeparam_k}\), then it has a nontrivial kernel for any specialization of \(\freeparam_1,\ldots,\freeparam_k\) in \(\RR\).
\end{lemma}
\begin{proof}\label{g:proof:idp28}
The kernel of a matrix \(A\) is nontrivial if and only if the matrix \(A^TA\) has zero determinant. Let \(A\) be a matrix with coefficients in \(\polyring{\freeparam_1,\ldots,\freeparam_k}\). If \(A\) has a nontrivial kernel over \(\fractionfield{\freeparam_1,\ldots,\freeparam_k}\), then
\begin{equation*}
\det(A^TA) = 0\in\polyring{\freeparam_1,\ldots,\freeparam_k}.
\end{equation*}
Consequently for any specialization \(\vectorparam{\freeparam}\in\RR^k\)
\begin{equation*}
\det(A(\vectorparam{\freeparam})^TA(\vectorparam{\freeparam})) = \det(A^TA)(\vectorparam{\freeparam}) = 0\in\RR,
\end{equation*}
so the specialized matrix \(A(\vectorparam{\freeparam})\) also has a nontrivial kernel.
\end{proof}
The above shows that for any choice of parameters \(\freeparam_{11},\freeparam_{12},\freeparam_{21},\freeparam_{22}\in\RR\), all trajectories to the origin for the linear ODE {({\ref{eq-linear-ode}})} lift to abnormals in the free Carnot group of rank 2 and step 7.

Finding an explicit covector also in the cases where the generic one vanishes would require a more careful look at the computation to solve the linear system over the fraction field \(\fractionfield{\vectorparam{\freeparam}}\). In the standard Gauss-Jordan algorithm, each attempted division by a nonconstant polynomial \(P\) splits the consideration into two cases: the points within the variety \(P=0\) and those outside. Inside the variety, the polynomial is replaced by zero and the Gauss-Jordan procedure continues. Outside the variety, the polynomial \(P\) can freely be used as a denominator. Expanding away the denominators in the end leads to finitely many different reduced echelon forms in distinct semialgebraic varieties, with each echelon form consisting of coefficients in the polynomial ring \(\polyring{\vectorparam{\freeparam}}\). Nontrivial elements of the solution space can then be read off from the echelon form, again splitting into cases based on the vanishing of the coefficients. This procedure defines an admissible abnormal covector \(\covector\) as a semialgebraic function \(\covector(\vectorparam{\freeparam})\).

However even in the relatively simple case of the linear ODE {({\ref{eq-linear-ode}})}, solving the resulting \(44\times 45\) system in nilpotency step 7 as described above leads to a rather cumbersome expression, and this will not be pursued here.

\typeout{************************************************}
\typeout{Subsection 5.3 The Hawaiian earring (r2s13)}
\typeout{************************************************}

\subsection{The Hawaiian earring (r2s13)}\label{example-the-hawaiian-earring}
\begin{figure}[hbtp]%
\centering%
\resizebox{0.325\linewidth}{!}{\begin{tikzpicture}%
\def\rmax{25}%
\draw[line width=0.25] (0,1) circle(1);%
\foreach \i in {2,...,\rmax}%
{%
\draw[line width={0.5/\i}] (0,{1/\i}) circle({1/\i});%
}%
\draw[line width=0.5/\rmax,fill] (0,{1/\rmax}) circle({1/\rmax});%
\end{tikzpicture}%
}%
\caption{The Hawaiian earring}%
\label{figure-hawaiian-earring}%
\end{figure}
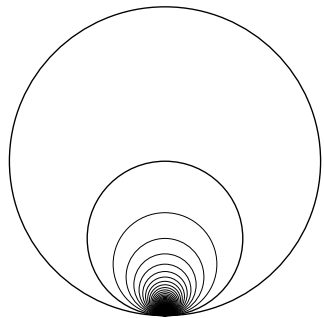
The Hawaiian earring is a countable union \(E=\bigcup_{n\in\NN}E_n\) of circles \(E_n = S^1((0,1/n),1/n)\) all with a common tangency point at 0. The punctured circles \(E_n\setminus\{(0,0)\}\) have parametrizations that are trajectories of the complex ODE \(\dot{z}=z^2\), i.e.\@, the planar quadratic ODE
\begin{align*}
\dot{x}_1 \amp= x_1^2-x_2^2,\qquad\\
\dot{x}_2 \amp= 2x_1x_2\text{.}
\end{align*}

Since all the punctured circles are trajectories of the same ODE and have a common point in their closure, the construction of {Algorithm~{\ref{algorithm-abnormal-trajectory}}} implies they all have the same covector. The concatenation resulting in the full Hawaiian earring will be shown to lift to an abnormal curve in the free Carnot group of rank 2 and step 13.

Step \ref{algorithm-1-orthogonal-poly} Set \(X_1Q:=2x_1x_2\) and \(X_2Q:=-x_1^2+x_2^2\).

Step \ref{algorithm-2-nonzero-derivatives} The nonzero commutators are
\begin{align*}
X_{12}Q \amp = [X_1,X_2]Q = X_1(X_2Q)-X_2(X_1Q) = -4x_1\\
X_{112}Q \amp = [X_1,X_{12}]Q = X_1(X_{12}Q)-X_{12}(X_1Q) = -4\text{.}
\end{align*}

Step \ref{algorithm-3-pde-integration} The {action} \(\freelie{2}\acts \polyring{x_1,x_2,x_{12},x_{112}}\) is
\begin{align*}
X_1 \amp= \partial_1,\amp
X_2 \amp= \partial_2+x_1\partial_{12}+\frac{1}{2}x_1^2\partial_{112},\amp
X_{12} \amp= \partial_{12}+x_1\partial_{112},\amp
X_{112} \amp= \partial_{112}\text{.}
\end{align*}
Integrating in \(x_{112}\) and then \(x_2\) and \(x_1\) gives the solution
\begin{equation*}
Q = x_1^2x_2+\frac{1}{3}x_2^3 - 4x_{112}\text{.}
\end{equation*}

Step \ref{algorithm-4-sufficient-step} Set \(m=4\). The dimensions of the first four layers \(\freelielayer{r}{k}\) are \(2,1,2,3\). The generating function determining the step upper bound is
\begin{align*}
\Delta(t) \amp = \frac{(1-3(1-t^3))t^3}{(1-t)^2(1-t^2)(1-t^3)^2}\\
\amp=-2 t^{4} - \cdots -20 t^{30} +45 t^{31} + \cdots +930 t^{38} + \cdots\text{.}
\end{align*}
This series has the partial sums \(\sum_{k=0}^{37}\differenceinteger{k} = -205\) and \(\sum_{k=0}^{38}\differenceinteger{k} = 725\), so the naive upper bound for the abnormality step is \(s=38\).

Step \ref{algorithm-5-abnormal-polys} Since the upper bound \(s=38\) is so large, it is more practical to solve the system in step \(s'\lt s\) and keep increasing \(s'\) until a nontrivial solution is found. Let \(\mathfrak{g}\) be the quotient \(\mathfrak{g}=\freelie{2,s'}/[\lowercentralseriesterm{\freelie{2,s'}}{4},\lowercentralseriesterm{\freelie{2,s'}}{4}]\). As before, the {abnormal polynomials} are computed by computing normal forms for the brackets
\begin{equation*}
\ad{X_{212}}^e\ad{X_{112}}^d\ad{X_{12}}^c\ad{X_2}^b\ad{X_1}^aX_{w_{i}},\quad i=1,2,3\text{,}
\end{equation*}
where \(w_{1}=1112\), \(w_{2}=2112\) and \(w_{3}=2212\) are the {deg-left-right Hall words} of degree 4. The normal forms are computed using the represention \(\ad{}\colon\mathfrak{g}\to\mathfrak{gl}(\lowercentralseriesterm{\mathfrak{g}}{4})\), where the only nontrivial commutators come from \([X_1,X_2]=X_{12}\), \([X_1,X_{12}]=X_{112}\) and \([X_2,X_{12}]=X_{212}\).

Step \ref{algorithm-6-linear-system} The first nontrivial solution is found in step \(s'=13\). The result is that the Hawaiian earring has an abnormal lift in the free Carnot group of rank 2 and step 13 with the covector \(\covector\) whose nonzero components are
\begin{align*}
\covector_{(2)^4(1)^5(1112)} \amp = 210\amp
\covector_{(112)^2(2)^2(1)(1112)} \amp = 2\\
\covector_{(2)^6(1)^3(1112)} \amp = 150\amp
\covector_{(2)^9(2112)} \amp = 280\\
\covector_{(2)^8(1)(1112)} \amp = 140\amp
\covector_{(112)(2)^6(2112)} \amp = -15\\
\covector_{(112)(2)^3(1)^3(1112)} \amp = -15\amp
\covector_{(112)^2(2)^3(2112)} \amp = 2\\
\covector_{(112)(2)^5(1)(1112)} \amp = -10\amp
\covector_{(112)^3(2112)} \amp = -2\text{.}
\end{align*}
The resulting {abnormal polynomials} have the factorizations
\begin{align*}
\abnormalpolynomial{1112}{\covector} \amp = -\frac{1}{96} x_{1}  x_{2}^{2} (-7 x_{1}^{2} x_{2} - x_{2}^{3} + 12 x_{112})Q\\
\abnormalpolynomial{2112}{\covector} \amp = \frac{1}{432} (-x_{2}^{3} + 3 x_{112}) (-15 x_{1}^{2} x_{2} - x_{2}^{3} + 12 x_{112})Q\\
\abnormalpolynomial{2212}{\covector} \amp = -\frac{1}{288} x_{1} (3 x_{1}^{4} x_{2} - x_{1}^{2} x_{2}^{3} - 6 x_{2}^{5} - 24 x_{1}^{2} x_{112} + 36 x_{2}^{2} x_{112})Q\text{.}
\end{align*}

\typeout{************************************************}
\typeout{Subsection 5.4 The Lorenz butterfly (r3s13)}
\typeout{************************************************}

\subsection{The Lorenz butterfly (r3s13)}\label{example-lorenz-butterfly}
\begin{figure}[hbtp]%
\centering%
\includegraphics[width=0.33\linewidth]{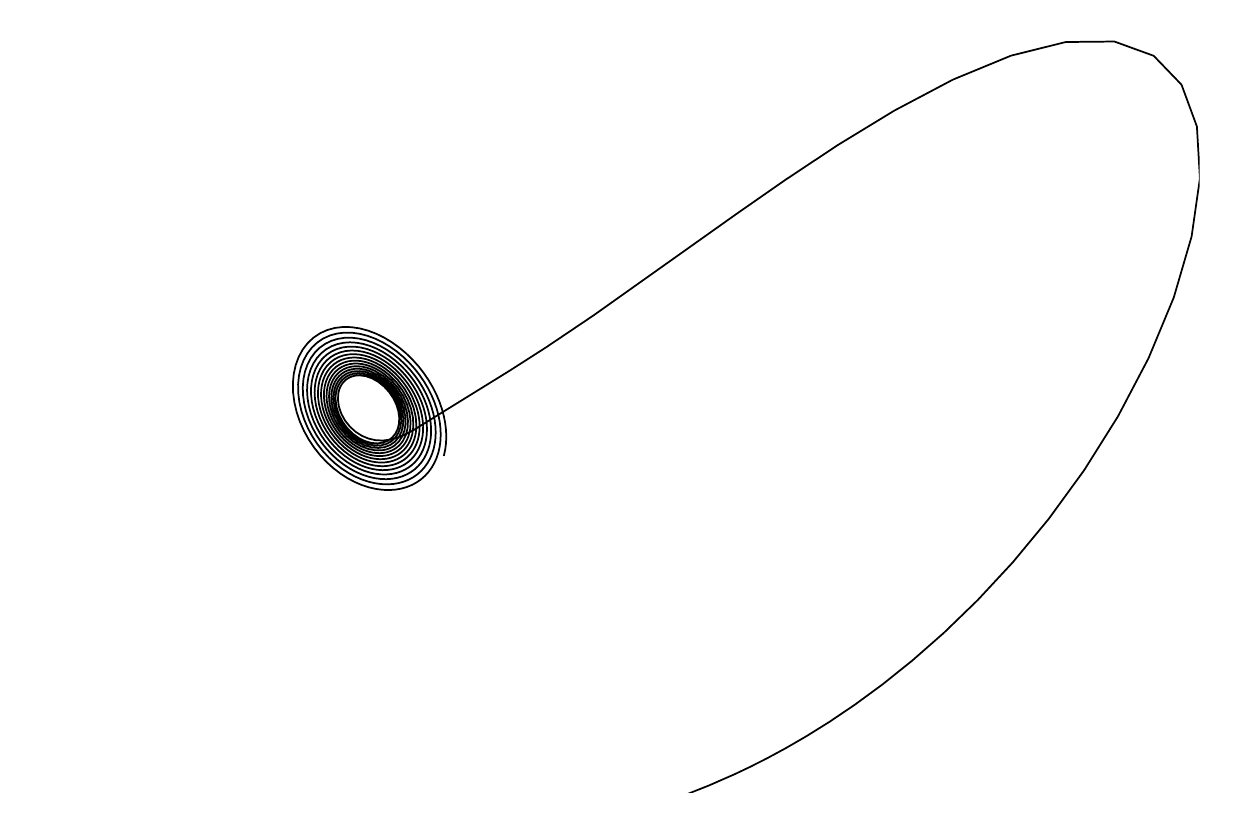}%
\includegraphics[width=0.33\linewidth]{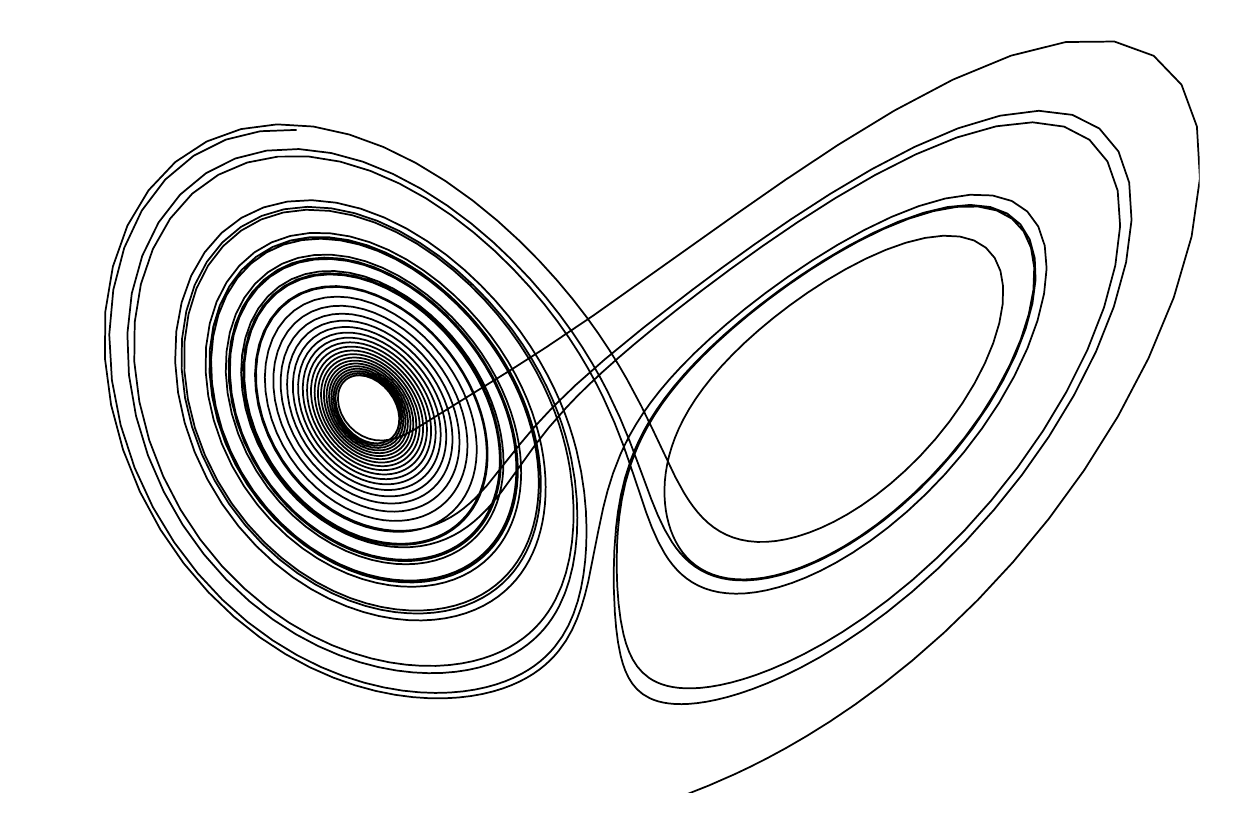}%
\includegraphics[width=0.33\linewidth]{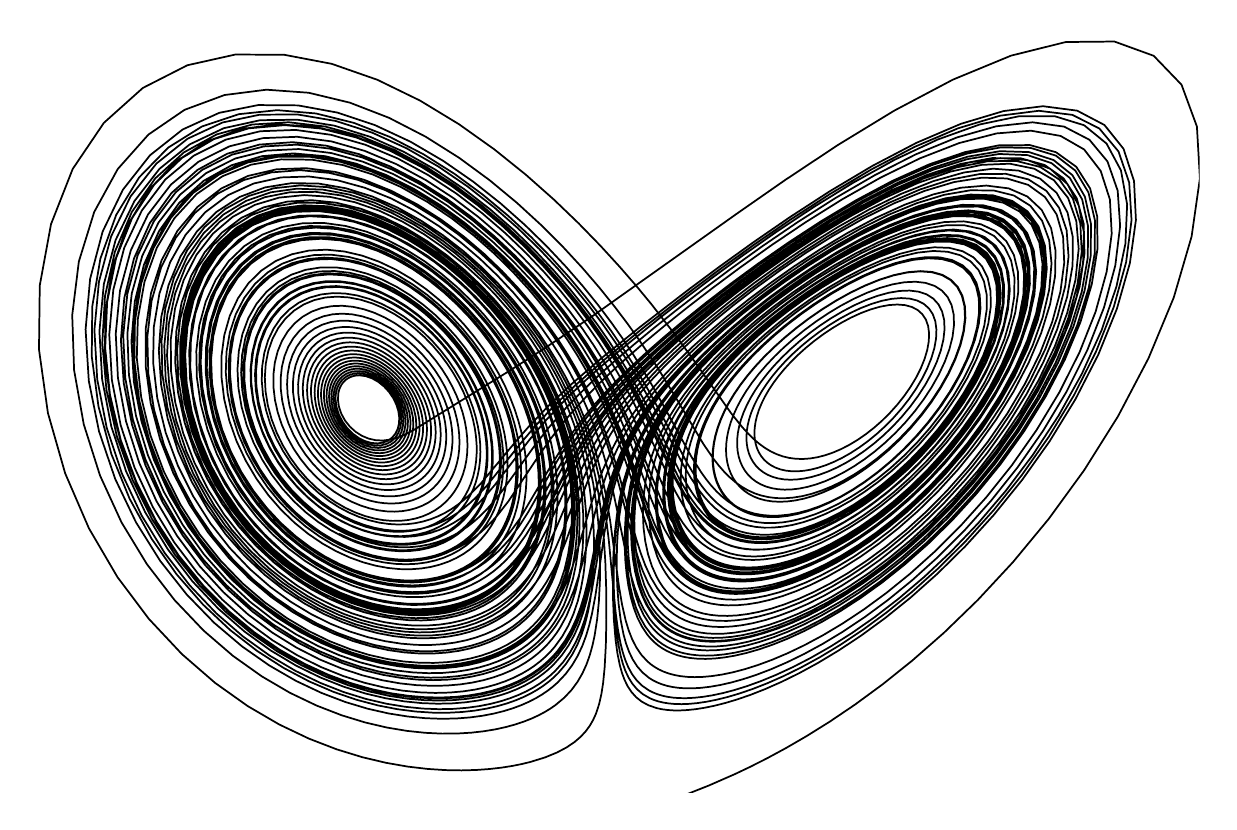}%
\caption{The trajectory of the Lorenz system starting from \((1,0,0)\) on the time interval \([0,T]\) for \(T=10\), \(T=30\), and \(T=100\).}%
\label{figure-lorenz-trajectory}%
\end{figure}
The Lorenz system
\begin{align*}
\dot{x}_1 \amp= 10(x_2-x_1)\\
\dot{x}_2 \amp= 28x_1-x_2-x_1x_3\\
\dot{x}_3 \amp= x_1x_2-\frac{8}{3}x_3
\end{align*}
is a classical example of a polynomial ODE system that exhibits chaotic behavior. Consider the trajectory starting from the point \((x_1,x_2,x_3)=(1,0,0)\), see {Figure~{\ref{figure-lorenz-trajectory}}} for a visualization. Following {Algorithm~{\ref{algorithm-abnormal-trajectory}}} will show that this trajectory lifts to an abnormal curve in the free Carnot group of rank 3 and step 13. Translating the initial point of the trajectory to the origin means that the ODE to study is
\begin{align*}
\dot{x}_1 \amp= -10x_1+10x_2+10\\
\dot{x}_2 \amp= -x_1x_3+28x_1-x_2+x_3-28\\
\dot{x}_3 \amp= x_1x_2-x_2-\frac{8}{3}x_3\text{.}
\end{align*}

Step \ref{algorithm-1-orthogonal-poly} In rank bigger than 2 there is more freedom to find a polynomial \(Q\in\polyring{\hallset}\) whose horizontal gradient is orthogonal to the ODE. One possible choice is
\begin{align*}
X_1Q\amp:=-x_1x_3 + 28x_1 - x_2 + x_3 - 28\\
X_2Q\amp:=10x_1 - 10x_2 - 10\\
X_3Q\amp:=0\text{.}
\end{align*}

Step \ref{algorithm-2-nonzero-derivatives} The {action} \(\freelie{3}\acts\polyring{x_1,x_2,x_3}\) on the horizontal variables is again \(X_1=\partial_1\), \(X_2=\partial_2\), \(X_3=\partial_3\), and all other elements give the zero derivation. The nonzero higher order derivatives of \(Q\) are
\begin{align*}
X_{12}Q \amp = X_1(X_2Q)-X_2(X_1Q) = 11\\
X_{13}Q \amp = X_1(X_3Q)-X_3(X_1Q) = x_1-1\\
X_{113}Q \amp = [X_1,X_{13}]Q = X_1(X_{13}Q)-X_{13}(X_1Q) = 1\text{.}
\end{align*}

Step \ref{algorithm-3-pde-integration} The {action} \(\freelie{3}\acts \polyring{x_1,x_2,x_{12},x_{13},x_{113}}\) is
\begin{align*}
X_1 \amp= \partial_1
\amp X_{12} \amp= \partial_{12}\\
X_2 \amp= \partial_2+x_1\partial_{12}
\amp X_{13} \amp= \partial_{13}+x_1\partial_{113}\\
X_3 \amp= \partial_3+x_1\partial_{13}+\frac{1}{2}x_1^2\partial_{113}
\amp X_{113} \amp= \partial_{113}\text{.}
\end{align*}
Integrating the variables in the order \(x_{113},x_{13},x_{12},x_3,x_2,x_1\) gives the solution
\begin{equation*}
Q = -\frac{1}{2}x_1^2x_3 + x_{113} + 14x_1^2 - x_1x_2 - 5x_2^2 + x_1x_3 + 11x_{12} - x_{13} - 28x_1 - 10x_2\text{.}
\end{equation*}

Step \ref{algorithm-4-sufficient-step} The dimensions of the first four layers of the free Lie algebra \(\freelie{3}\) are \(3,3,8,18\). The generating function
\begin{equation*}
\Delta(t) = \frac{(1-18(1-t^3))t^4}{(1-t)^3(1-t^2)^3(1-t^3)^8}
\end{equation*}
gives the a priori bound of step \(s=724\) for when a lift becomes abnormal.

Step \ref{algorithm-5-abnormal-polys} With the unreasonably large nilpotency step \(s=724\), it is better to solve the system in step \(s'\lt s\) and keep increasing \(s'\) until a nontrivial solution is found. In the quotient \(\mathfrak{g}=\freelie{3,s'}/[\lowercentralseriesterm{\freelie{3,s'}}{4},\lowercentralseriesterm{\freelie{3,s'}}{4}]\), the {abnormal polynomials} can be computed via normal forms for Lie brackets using the smaller family of commutation rules of the restricted adjoint representation \(\ad{}\colon\mathfrak{g}\to\mathfrak{gl}(\lowercentralseriesterm{\mathfrak{g}}{4})\) as in the previous examples.

Step \ref{algorithm-6-linear-system} The first solution exists in step \(s=13\), so the trajectory starting from \((1,0,0)\) of the Lorenz system is abnormal in the free Carnot group of rank 3 and step 13. In rank 3 and step 13 the linear system consists of 81360 equations in 34465 variables \(\covector\) and 9918 variables \(\factorcoefficient{}\). The solutions are however relatively sparse, with an example solution having 476 nonzero coefficients out of the possible 34465, suggesting that further simplifications using more refined quotient Lie algebras could be possible.

\bibliographystyle{amsalpha}
\bibliography{biblio}

\end{document}